\documentclass[11pt]{article}
\usepackage{amsmath,amsfonts,amsthm,amscd,amssymb,graphicx}
\usepackage{subfigure}

\numberwithin{equation}{section}

\usepackage{hyperref}
\usepackage{verbatim}
\usepackage{color}
\newtheorem{theorem}{Theorem}[section]

\newtheorem{lemma}[theorem]{Lemma}

\newtheorem{proposition}[theorem]{Proposition}
\newtheorem{prop}[theorem]{Proposition}

\def\eps{\varepsilon }

\newcommand{\RR}{\mathbb{R}}

\newcommand{\CC}{\mathbb{C}}

\newcommand{\ZZ}{{\mathbb Z}}
\newcommand{\TT}{{\mathbb T}}

\def\beq{\begin{equation}}
\def\eeq{\end{equation}}
\def\bb1{{1\!\!1}}

\def\cL{\mathcal{L}}
\def\cW{\mathcal{W}}

\def\R{\mbox{Re }}

\def\w{{\omega}}

\def\pt{\partial}

\newcommand{\T}{{\mathbb T}}

\def\eps{\varepsilon}

\def\triangle{\Delta}
\def\bega{\begin{aligned}}
\def\enda{\end{aligned}}
\def\R{\mathbb{R}^2}

\def\lw{\left}
\def\rw{\right}

\def\R{\mathbb{R}}
\def\wtd{\widetilde}

\def\Z{\mathbb{Z}}
\def\bcase{\begin{cases}}
\def\ecase{\end{cases}}
\def\al{\alpha}
\def\bmx{\begin{bmatrix}}
\def\emx{\end{bmatrix}}

\def \R{{\mathbb{R}}}

\def \Z{{\mathbb{Z}}}

\def \T{{\mathbb{T}}}

\def \del {\partial}

\begin{document}

\title{The inviscid limit for the $2d$ Navier-Stokes equations in bounded domains}
\author {Claude Bardos\footnotemark[1] \and
Trinh T. Nguyen\footnotemark[2]
\and Toan T. Nguyen\footnotemark[3]
\and Edriss S. Titi\footnotemark[4]
}

\maketitle

\renewcommand{\thefootnote}{\fnsymbol{footnote}}

\footnotetext[1]{Laboratoire J.-L. Lions, Sorbonne Universit\'e,
BC 187, 4 place Jussieu, 75252 Paris, Cedex 05, France. Email: claude.bardos@gmail.com
}

\footnotetext[2]{Department of Mathematics, University of Southern California, LA, CA 90089.
Email: tnguyen5@usc.edu. TN is partly supported by the AMS-Simons Travel Grant Award.
}

\footnotetext[3]{Department of Mathematics, Penn State University, State College, PA 16803. Email: nguyen@math.psu.edu. TN is partly supported by the NSF under grant DMS-2054726.}

\footnotetext[4]{Department of Mathematics, Texas A\&M University, 3368 TAMU, College Station, TX 77843-3368, USA. Department of Applied Mathematics and Theoretical Physics, University of Cambridge, Cambridge CB3 0WA, U.K. ALSO, Department of Computer Science and Applied Mathematics, Weizmann Institute of Science, Rehovot 76100, Israel.
Emails: titi@math.tamu.edu; Edriss.Titi@maths.cam.ac.uk

}

\begin{center}
\em In memory of Robert T. Glassey
\end{center}

\bigskip

\begin{abstract}

We prove the inviscid limit for the incompressible Navier-Stokes equations for data that are analytic only near the boundary in a general two-dimensional bounded domain. Our proof is direct, using the vorticity formulation with a nonlocal boundary condition, the explicit semigroup of the linear Stokes problem near the flatten boundary, and the standard wellposedness theory of Navier-Stokes equations in Sobolev spaces away from the boundary.

\end{abstract}

\tableofcontents

\section{Introduction}

We are interested in the inviscid limit of solutions to the incompressible Navier-Stokes equations (NSE)
\begin{equation}\label{NS}
\begin{aligned}
\partial_t u + u \cdot \nabla u + \nabla p &= \nu \Delta u,
\\
\nabla \cdot u & = 0,
\end{aligned}
\end{equation}
in a bounded domain $\Omega \subset \RR^2$, with initial data $u_{\vert_{t=0}} = u_0(x)
$ and with the no-slip boundary condition
\begin{equation}\label{NS-BC}
u|_{\partial \Omega} =0.
\end{equation}
In the inviscid limit: $\nu \to 0$, one would intuitively expect  that the solutions $u_\nu$, of problem \eqref{NS}-\eqref{NS-BC}, converge to the corresponding solutions of the Euler equations of ideal incompressible fluids
\begin{equation}\label{Euler}
\begin{aligned}
\partial_t u + u \cdot \nabla u + \nabla p &= 0, \quad \hbox{in} \quad \Omega,
\\
\nabla \cdot u & = 0, \quad \hbox{in} \quad \Omega,
\\
u\cdot n& = 0, \quad \hbox{on} \quad \partial \Omega,
\end{aligned}
\end{equation}
where $n$ denotes the unit normal vector to the boundary pointing inward. However, the inviscid limit for problem \eqref{NS}-\eqref{NS-BC} is strenuous and remains open due to the appearance of boundary layers and strong shear near the boundary that  triggers the shedding of unbounded vorticity by the boundary. In their celebrated work \cite{CS98}, Caflisch and Sammartino establish the boundary layer expansion and the inviscid limit for analytic data on the half-plane. Maekawa \cite{Maekawa} proved a similar result that allows Sobolev data whose vorticity is supported away from the boundary. The result and its proof was recently simplified \cite{NN18} and extended in \cite{KVW20,KNVW}, which allow data that are only analytic near the boundary.

In this paper, we prove the inviscid limit of \eqref{NS}-\eqref{NS-BC} for data that are only analytic near the boundary of a general bounded analytic domain in $\RR^2$, thus further extending \cite{CS98, Maekawa, NN18, KVW20} from the case of half-plane to bounded domains with analytic boundaries. Precisely, we assume that

\begin{itemize}

\item $\Omega$ is a simply-connected bounded domain in $\RR^2$ whose boundary $\del\Omega$ is an analytic curve, defined by an analytic map: $\theta \in \T= \R/(\Z L)  \mapsto  x(\theta) =(x_1(\theta),x_2(\theta) ) \in \del\Omega\,.$

\end{itemize}

The analyticity of the boundary naturally extends to an analytic map which maps the near-boundary part of the domain $\{x\in \Omega: \,\,d(x,\partial \Omega)<\delta\}$ to the case of half-plane $(z,\theta)\in  (0,\delta)\times \mathbb{T}$, where $z$ is the distance function from the boundary. Here, for sake of presentation, we have chosen to consider the case of simply-connected domain $\Omega$. The results of this paper apply to the general setting of multi-connected bounded domains whose boundaries consist of closed analytic curves, i.e., including domains with holes. Our analysis near each of the boundaries is close to that on the half-plane. A crucial assumption, however, lies on the analyticity of initial data near the boundary, which appears to be sharp.

{\em The work is dedicated to the memory of Professor Robert T. Glassey, who was a  great mathematician, a {close} friend, and an inspiring teacher.}

\subsection{Boundary vorticity formulation}\label{sec-vorticityBC}

We shall work with the boundary vorticity formulation \cite{Anderson,Maekawa,NN18}. Precisely, let $u=(u_1,u_2)$ be the velocity vector field and $\omega = \nabla^\bot \cdot u = \partial_{x_2} u_1 - \partial_{x_1} u_2$ be the corresponding vorticity. Then, the vorticity equation reads
\beq\begin{aligned}\label{NS-vor}
\partial_{t}\omega+ u\cdot\nabla\omega&=\nu\Delta\omega,\\
u &= \nabla^\perp \Delta^{-1} \omega, \qquad \hbox{(the Biot-Savart law).} 
\end{aligned}
\eeq
Here and throughout the paper, $\Delta^{-1}$ denotes the inverse of the Laplacian operator in $\Omega$ subject to the zero Dirichlet boundary condition. Evidently, this, together with  the Biot-Savart law, imply the impermeability boundary condition $u\cdot n=0$ on $\partial \Omega$.
To ensure the full no-slip boundary condition, i.e., that $u\cdot \tau=0$ on the boundary $\partial \Omega$, where $\tau$ in the unit tangent vector to the boundary, we first require that the initial data satisfy the no-slip boundary condition \eqref{NS-BC}, and then we impose in addition that  $\partial_t u\cdot \tau= 0$ on the boundary, $\partial \Omega$, for all positive time. This leads to the boundary condition
\beq \label{h-BC}
0 =\tau\cdot \partial_t u = \tau\cdot \nabla^\perp \Delta^{-1} \partial_t \omega = \partial_n [\Delta^{-1} (\nu \Delta \omega - u \cdot \nabla \omega)]
\eeq
on the boundary. Introduce $\omega^*$ to be the solution of the nonhomogeneous Dirichlet  boundary-value problem
\begin{equation}\label{def-omegas}\left\{
\begin{aligned}\Delta \omega^* &= 0, \qquad \mbox{in}\quad \Omega
\\
{\omega^*} &=\omega, \qquad \mbox{on }\quad \partial \Omega.
\end{aligned}
\right.
\end{equation}
and define the Dirichlet-Neumann operator by
\begin{equation}\label{def-DN}
DN \omega = -\partial_n \omega^*, \qquad \mbox{on}\quad \partial \Omega,
\end{equation}
where $\omega^*$ solves \eqref{def-omegas}. Observe that $\partial_n [\Delta^{-1}\Delta \omega] =\partial_n [\Delta^{-1}\Delta (\omega - \omega^*)]= (\partial_n + DN)\omega$. Thus, by virtue of the boundary condition \eqref{h-BC} the boundary condition on vorticity reads
\beq\label{BC-vor}
\nu (\partial_n + DN)\omega_{\vert_{\partial \Omega}} = [\partial_n \Delta^{-1} ( u \cdot \nabla \omega)]
_{\vert_{\partial \Omega}},
\eeq
together with the Biot-Savart law \eqref{NS-vor}.

Throughout this paper, we shall deal with the Navier-Stokes solutions that solve \eqref{NS-vor}-\eqref{def-DN}, or equivalently \eqref{NS-vor} and \eqref{BC-vor}. Such a solution will  be constructed via the Duhamel's integral representation, treating the nonlinearity as a source term. As we observed earlier the boundary condition $u\cdot n=0$ on $\partial \Omega$ follows from the Biot-Savart law and the definition of $\Delta^{-1}$ with the zero Dirichlet boundary condition.

\subsection{Main results}

Our main result reads as follows.

\begin{theorem} \label{theo-main} Let $u_0\in H^5(\Omega)$ be an initial data that vanishes on the boundary. We assume that the initial vorticity $\omega_0$ is analytic near the boundary $\partial \Omega$ (see Section \ref{sec-norm}). 
 Then, there is a positive time $T$, independent of $\nu$, so that the unique solution $u_\nu(t)$ to the Navier-Stokes problem \eqref{NS}-\eqref{NS-BC}, for every $\nu >0$, with initial data $u_0$,  exists on $[0,T]$ and {has vorticity $\omega_\nu = \nabla^\perp\cdot u_\nu$ that remains analytic near the boundary, and satisfies} 
\begin{equation}\label{vorticity-bd}
\lim_{\nu\rightarrow 0} \sqrt\nu \|\omega_\nu\|_{L^\infty([0,T]\times\partial \Omega))} < \infty.
\end{equation}
Moreover, in the inviscid limit as $\nu \to 0$, $u_\nu$ converges strongly in $L^\infty([0,T];L^p(\Omega))$, for any $2\le p<\infty$, to the corresponding solution $u$ of the Euler equations \eqref{Euler} with the same initial data $u_0$. 
\end{theorem}

{The fact that Euler solutions remain analytic near the boundary is a classical result \cite{BB,KV11}, which is a direct consequence of the main theorem.
}
The main difficulty in establishing the inviscid limit is to control the vorticity on the boundary and derive uniform estimates such as \eqref{vorticity-bd}, which is the main contribution of this paper. The inviscid limit then follows easily. In fact, a much weaker bound than \eqref{vorticity-bd} is sufficient to guarantee the convergence of solutions to the Navier-Stokes to a corresponding solution of the Euler equations. Precisely, we have the following simple Kato's type theorem.

\begin{theorem} \label{theo-Kato} Let $T>0$ and  $u$ be a weak solution to the Euler equations \eqref{Euler} in $[0,T]\times \Omega$  satisfying $\| \nabla u \|_{L^\infty([0,T]\times \Omega) } <\infty$. Suppose that, for every $\nu >0$, $u_\nu$ are  Leray weak solutions to the Navier-Stokes problem \eqref{NS}-\eqref{NS-BC} on $[0,T]\times \Omega$,  satisfying
\begin{equation}\label{NSE-energy}
\sup_{0<t<T}  \|u_\nu(t)\|^2_{L^2(\Omega)} + \nu\int_0^T\|\nabla_x u_\nu(t)\|_{L^2(\Omega)}dt\le C_0,
\end{equation}
uniformly in $\nu \to 0$. Assume that the vorticity  $\omega_\nu=\nabla^\bot\cdot u_\nu$ satisfies
\begin{equation}\label{omegagronwall}
\limsup_{\nu \to 0}  \Big ( - \int_0^T\int_{\del\Omega}   \nu \omega_\nu(t,\sigma)   u(t,\sigma) \cdot \tau(\sigma)  d\sigma dt\Big) =0,
\end{equation}
then any $\overline u_\nu$, which is a {\bf weak$-*$ limit} in $L^\infty([0,T]; L^2(\Omega))$  of a subsequence $u_{\nu_j}$ of the Leray weak solutions, as $\nu_j \to 0$, satisfies the stability estimate:
\begin{equation}\label{Stability}
 \|\overline{u_\nu(t)}-u(t)\|_{L^2(\Omega)}^2 \le e^{2t \| \nabla u \|_{L^\infty([0,T]\times \Omega)}} \|\overline{u_\nu(0)}-u(0)\|_{L^2(\Omega)}^2 .
\end{equation}
In particular, if $u_\nu(0)\rightarrow u(0)$ in $ L^2(\Omega)$, as $\nu \to 0$, then $u_\nu$ converges strongly to $u$ in $L^\infty([0,T];L^2(\Omega)).$ 
\end{theorem}
\begin{proof} An elementary manipulation (e.g., \cite{BT13}) yields the following energy inequality
\begin{equation}
\begin{aligned}
&\|u_\nu(t)- u(t) \|_{L^2(\Omega)}^2 +   \nu  \int_0^t  \|\nabla u_\nu(s) \|^2_{L^2(\Omega)} ds
\\&\le
   \|u_\nu(0) -u(0)\|_{L^2(\Omega)}^2 +   \nu   \int_0^t \|\nabla u (s) \|^2_{L^2(\Omega)} ds
-  \int_0^t \int_{\del\Omega}\nu (\partial_n u_\nu (s,\sigma) ) \cdot u(s,\sigma) \; d\sigma ds
 \\& + \int_0^t\int_{\Omega}|\Big((  \nabla u +\nabla^\bot u  ) (u_\nu - u)\Big)\cdot(u_\nu-u))| dx ds
 \\&\le
   \|u_\nu(0) -u(0)\|_{L^2(\Omega)}^2 +   \nu   \int_0^t \|\nabla u (s) \|^2_{L^2(\Omega)} ds
-  \int_0^t \int_{\del\Omega}\nu \omega_\nu (s,\sigma) (u(s,\sigma)\cdot \tau(\sigma)) \; d\sigma ds
 \\& + 2 \| \nabla u \|_{L^\infty([0,T]\times \Omega) }\int_0^t  \|u_\nu(s)- u(s) \|_{L^2(\Omega)}^2  \, ds,
\end{aligned}
\end{equation}
where in the third term in the right-hand side of the last inequality we used the fact that $(\partial_n u_\nu  ) \cdot u = \omega_\nu (u \cdot \tau)$ on the boundary. 
Let  $u_{\nu_j}$ be a subsequence which  converges  weak$-*$ in $L^\infty([0,T]; L^2(\Omega))$, as $\nu_j \to 0$. We apply the above energy inequality to $u_{\nu_j}$ and  invoke Gronwall's Lemma. Observe that since the Leray weak solutions belong to $C([0,T]; L^2(\Omega))$ then $\|u_\nu(0)\|^2_{L^2(\Omega)}\le C_0$ by virtue of \eqref{NSE-energy}. Thanks to the Banach-Alaoglu Theorem and  assumption \eqref{omegagronwall} we conclude \eqref{Stability}.  The last part of the theorem is an immediate consquence of \eqref{Stability}.
\end{proof}

\subsection{Remarks}

As mentioned in the introduction, our main results extend the previous works \cite{CS98, Maekawa, NN18, KVW20} from the case of the half-plane to bounded domains. The analyticity near the boundary is required to control the unbounded vorticity in the inviscid limit. It may be possible to extend the present analysis to include the propagation of boundary layers and the classical Prandtl's boundary layer expansions, whose validity near general boundary layers again requires analyticity.

The first such a result was due to the celebrated work by Asano \cite{Asano} and Sammartino-Caflisch \cite{CS98}, where the boundary layer expansion was established for data on the half-plane that are analytic in both horizontal and vertical variables. When constructing solutions to the Prandtl equation, the analyticity in the vertical variable can be dropped \cite{CanS}. It is not known however if such an assumption can be dropped at the level of Navier-Stokes equations. Maekawa \cite{Maekawa} established the Prandtl's expansion for data whose vorticity is compactly supported away from the boundary, while recently Kukavica, Nguyen, Vicol and Wang \cite{KNVW} extended the result to include data that are analytic only near the boundary, building upon the vorticity formulation revived by Maekawa \cite{Maekawa}, the direct proof of the inviscid limit for analytic data developed in Nguyen and Nguyen \cite{NN18}, and the Sobolev-analytic norm developed in Kukavica, Vicol and Wang \cite{KVW20}. All these aforementioned works are on the half-plane. We mention a recent result \cite{WW20}, which to the best of our knowledge was the first to establish a Prandtl asymptotic expansion in a curved domain.

When background boundary layers have no inflection point, the analyticity can be relaxed to include perturbations in Gevrey-$\frac32$ spaces \cite{GMM1,GMM2}, which is sharp in view of the Kelvin-Helmholtz type of instability of generic boundary layers and shear flows \cite{GGN1,GGN2}. When Sobolev data is allowed, the Prandtl's asymptotic expansion is false due to counter-examples given in \cite{Grenier, GrenierNguyen,GrenierNguyen2}, where the failure of the convergence from Navier-Stokes to Euler solutions, plus a Prandtl corrector, is due to an emergence of viscous boundary sublayers that reach to order one, independent of viscosity, in $L^\infty$ norm for velocity \cite{GrenierNguyen}.

\section{Navier-Stokes equations near the boundary}

\subsection{Global geodesic coordinates}\label{sec-geodesic}

Following a construction done in \cite{BT21} we introduce a   well adapted representation   of $\del\Omega\,,$
 $$\theta \in \T= \R  \slash (\Z L)  \mapsto  x(\theta) =(x_1(\theta),x_2(\theta) ) \in \del\Omega$$
 which, being global, preserves  the analyticity hypothesis. Let $\vec \tau(\theta) $ and $\vec n(\theta) $ be the unit tangent and interior normal vectors at the boundary:
 \begin{equation}
  \begin{aligned}
  & \vec \tau (\theta)=\vec \tau(x(\theta)  =(  x'_1(\theta),   x'_2(\theta) ), \quad \hbox{ and} \quad \vec n (\theta) =\vec n (x(\theta))=  (- x'_2 (\theta),  x'_1(\theta) )\\
  & \hbox{ with} \quad  | x'(\theta)|^2= ( x'_1(\theta) )^2 +( x'_1(\theta) )^2 =1.
\end{aligned}
   \end{equation}
Let $d(x,\del \Omega)$ denotes the distance of any point $x\in \R^2 $ to $\del \Omega$ . Then  we have the following classical result.

\begin{proposition}\label{prop-deltad}
There exists a $ \delta >0$ such that for each $x$ on the open set
 \begin{equation} \label{delta}
 V_{  \delta}=\{x\in\R^2\quad \hbox{with} \quad  d(x,\del\Omega)< {\delta}\}
 \end{equation}
   there  is a unique point $\hat x(\theta)  \in \del \Omega$
   with
 $d(x,\del\Omega) =|x-\hat x(\theta) |.$ The mapping
   $x\mapsto \hat x(\theta) $ is an analytic map from  $   V_\delta $ with value in $ \del \Omega $. In addition, for $x\in V_\delta\,,$  one has the formula
   \begin{equation}
   \nabla_x d(x,\del\Omega) =\vec n (x(\theta)).
   \end{equation}
   \end{proposition}

When no confusion is possible, for $x\in V_\delta$ the notations $\vec n(x)$ and $\vec\tau (x) $  will be used for $\vec n(x(\theta)) $  and $\vec \tau (x(\theta)) $ respectively.
Observe that 						
\begin{equation}
 \vec\tau'(\theta) \wedge \vec n(\theta)=  x'_1 (\theta) x_1''(\theta)  + x'_2(\theta)  x_2''(\theta) =\frac{d}{ d \theta}| x'(\theta)|^2=0\,,
\end{equation}
which implies the relation
\begin{equation}
\vec n'(\theta) = \gamma(\theta) \vec \tau (\theta) \quad\hbox{and} \quad \vec \tau'(\theta)= \gamma(\theta) \vec n(\theta)\,,
\end{equation}
with
\begin{equation}
\gamma(\theta)= x_1''(\theta) x_2'(\theta) -x'_1(\theta) x_2''(\theta),
\end{equation}
being the curvature of the boundary $\del \Omega \,.$ Therefore  the mapping:
\begin{equation}
(z,\theta)\mapsto  X(z,\theta) =  x(\theta)  + z\vec n(  x(\theta)),
\end{equation}
defines au global $C^2$ diffeomorphisme of $[-\delta ,\delta ]\times (\R/(L\Z)) $ on $\overline{V_\delta}\,.$ Moreover,   for any vector field $x\in \overline \Omega \mapsto v(x)\,,$ as soon as $x\in\overline V_\delta\,,$ using the above notations, one has:
\begin{equation}
v(x)= (v(x)\cdot\vec \tau(x)) \vec \tau(x) +(v(x)\cdot\vec n(x)) \vec n  (x) \,.
\end{equation}
Below, for sake of clarity,   the symbol $X$ is  used for any $x=X(z,\theta)$. There hold
\begin{equation} \label{isomorph}
\begin{aligned}
&\del_z X(z,\theta) =\vec n(\theta)
\,,\quad \del_\theta X(z,\theta) = J(z,\theta) \vec \tau (\theta)\,,\\
&\hbox{and}\quad  J(z,\theta)= 1 + z\gamma(\theta) >0  \quad \hbox{for}\quad |z|<\delta\,,
\end{aligned}
\end{equation}
provided $\delta >0$ is chosen to be small enough.
From the relation
\begin{equation}
                     \begin{pmatrix}  \del_z X_1  & \del_\theta  X_1  \\
                                       \del_z X_2  & \del_\theta  X_2    \end{pmatrix}
                     \begin{pmatrix}  \del_{X_1} z  &   \del_{X_2} z  \\
                                      \del_{X_1} \theta   &   \del_{X_2} \theta \end{pmatrix}
                                      =
                     \begin{pmatrix}  1 &  0  \\
                                      0 &  1 \end{pmatrix}\,,
 \end{equation}
one deduces the formula:
\begin{equation} \label{gradients}
\nabla_X\theta  =\frac {\vec \tau ( z,\theta))}{  J(z,\theta) } \quad \hbox{and}\quad
 \nabla_X z  = \vec n(   \theta ) \,.
\end{equation}
We collect the following useful relations whose derivations are classical. For any vector field $u$, we have
\begin{equation}\label{sdiv}
\begin{aligned}
 \nabla \cdot u &=\frac1J(\del_z ( J(u\cdot \vec n)) +\del_\theta(u\cdot \vec \tau))  = \del_z  (u\cdot \vec n) +\frac1J \del_\theta(u\cdot \vec \tau)) + \frac\gamma J u\cdot \vec n \,,
\\
  \nabla \wedge u &= \frac1J(\del_z ( J u\cdot \vec \tau) -\del_\theta(u\cdot \vec \tau))=
   \del_z  ( u\cdot \vec \tau) -\frac1J \del_\theta(u\cdot \vec \tau)) +\frac\gamma J ( u\cdot \vec \tau)  \,.
\end{aligned}
  \end{equation}
  For any scalar function $\Psi$, we have
  \begin{equation}\label{ssvort}
  \nabla\wedge \Psi =\frac1J  \Bigg(
  \begin{aligned}
  &  \del_z(J\Psi)\\
    & -\del_\theta \Psi
 \end{aligned}  \Bigg)=  \Bigg(
  \begin{aligned}
  &  \del_z  \Psi \\
    & -\frac1J\del_\theta \Psi
 \end{aligned}  \Bigg)+  \Bigg(
  \begin{aligned}
  &   \frac\gamma J \Psi \\
    &0
 \end{aligned}  \Bigg)\,,
  \end{equation}
 and
 \begin{equation}\label{Lap-sa}
 \begin{aligned}
 \Delta \Psi
 &= \frac 1J \del_z (J\del_z \Psi) + \frac1J \del_\theta (  \frac 1J\del_\theta \Psi)
=
 \Delta_{z,\theta} \Psi + R_\Delta\Psi\,,
 \end{aligned}
 \end{equation}
 in which we denote
 $$\Delta_{z,\theta}=(\del^2_z+\del^2_\theta), \quad R_\Delta  = m(z,\theta) \del^2_\theta +  \frac\gamma{1 + z\gamma} \del_z -\frac{z\gamma'}{(1 + z\gamma)^3} \del_\theta \quad \hbox{and} \quad m(z,\theta)= -\frac{2z\gamma +(z \gamma)^2}{(1+z\gamma)^2}.$$

\subsection{Scaled coordinates}

In view of \eqref{Lap-sa}, we observe that the Laplacian $\Delta$ is nearly the flat Laplacian $ \Delta_{z,\theta}$, in the $(z,\theta)$ coordinates, near the boundary. To make use of this fact, we introduce the following scaled variables
\begin{equation}
(\wtd z, \wtd \theta) = (\lambda z, \lambda \theta)
\end{equation}
for sufficiently small $\lambda \in (0,1)$. By construction, we compute
\begin{equation}
\label{Lap-scaled}
\Delta = \lambda^2 \Big(  \triangle_{\wtd z,\wtd \theta} + \lambda^2 \widetilde  R_\Delta \Big)\,,
\end{equation}
in which $\Delta_{\wtd z,\wtd\theta}=(\del^2_{\wtd z}+\del^2_{\wtd \theta})$ and
\begin{equation}\label{def-RDelta}
\wtd R_\Delta  = \wtd m(\wtd z,\wtd \theta) \del^2_{\wtd \theta} +  \frac{\wtd \gamma}{1+\lambda^2\wtd z\wtd \gamma} \del_{\wtd z} -\frac{\wtd z\wtd \gamma'}{(1+\lambda^2\wtd z\wtd \gamma)^3} \del_{\wtd \theta}  \quad \hbox{and} \quad \wtd m(\wtd z,\wtd\theta)= -
 \frac{2\wtd z\wtd \gamma+\lambda^2(\wtd z\wtd \gamma)^2}{(1+\lambda^2\wtd z\wtd \gamma)^2}\,,
 \end{equation}
where $\gamma=\lambda^3\wtd \gamma(\wtd \theta)$. In the analysis, $\lambda$ will be taken sufficiently small, and so $\Delta$ is indeed approximated by $ \lambda^2 \triangle_{\wtd z,\wtd \theta}$, treating $\lambda^2 \widetilde  R_\Delta$ as a perturbation.

\subsection{Vorticity equations near the boundary}\label{sec-bdryvorticity}

In this section, we derive vorticity equations in the geodesic coordinates near the boundary in the region $V_\delta$ defined as in Proposition \ref{prop-deltad}. Introduce a smooth cutoff function $\phi^b(x)$ so that
\begin{equation}\label{def-phib}
\phi^b(x) = \left \{ \begin{aligned}
1, \qquad & \mbox{if}\quad \lambda d(x,\partial\Omega)\le \delta_0+\rho_0
\\
0, \qquad & \mbox{if} \quad \lambda d(x,\pt\Omega)\ge \delta_0+2\rho_0
\end{aligned} \right.
\end{equation}
for small positive constants $\delta_0, \rho_0$ so that $\delta_0 + 2\rho_0 < \lambda \delta$ to guarantee that $\hbox{supp} (\phi^b) \subset V_\delta$ as in Proposition \ref{prop-deltad}. Define
\begin{equation}\label{def-wb}
\w^b=\phi^b(x)\w(t,x) .
\end{equation}
It follows from \eqref{NS-vor} that
\beq \label{eq-near-bdr}
\pt_t \w^b-\nu\triangle \w^b= N^b,
\eeq
where
$$ N^b: =-u\cdot\nabla \w^b +(u\cdot\nabla \phi^b)\w-\nu(\triangle \phi^b)\w-2\nu \nabla\phi^b\cdot\nabla\w.
$$
Observe that $N^b(u,\omega) =0$ on $\{\lambda d(x,\pt\Omega)\ge\delta_0+2\rho_0\}$ where the cutoff function $\phi^b$ vanishes. We then introduce
the following scaled vorticity
\begin{equation}\label{vorticity-scale}
\omega^b(t,x)=\widetilde\omega(\lambda^2 t,\lambda \theta,\lambda z), \qquad (\wtd t,\wtd z,\wtd \theta) = (\lambda^2 t,\lambda z,\lambda \theta), 
\end{equation}
for small $\lambda>0$. Using \eqref{Lap-scaled}, we rewrite the vorticity equation as
\beq \label{nearBdr}
\lw(\pt_{\wtd t}-\nu \triangle_{\wtd z,\wtd \theta}\rw)\wtd \w= - \nu \lambda^2 \wtd R_\Delta \wtd \w + \lambda^{-2}N^b .
\eeq
Equation \eqref{nearBdr} is defined on $(\wtd z,\wtd \theta) \in \RR_+ \times \TT$ (in fact, the equation vanishes for $\wtd z \ge \delta_0 + 2\rho_0$).
We shall solve \eqref{nearBdr} together with the boundary condition \eqref{BC-vor}, which now reads
\begin{equation}\label{BC-vornear}
\nu (\partial_{\wtd z} + \widetilde{DN})\wtd \omega_{\vert_{\wtd z=0}} = \lambda^{-1}[\partial_n \Delta^{-1} ( u \cdot \nabla \omega)]
_{\vert_{\partial \Omega}} .
\end{equation}
System \eqref{nearBdr}-\eqref{BC-vornear} will be our main equation for the scaled vorticity near the boundary. Away from the boundary, we construct vorticity using the original system as derived in Section \ref{sec-vorticityBC}.

\subsection{Dirichlet-Neumann operator}

Let us precise the Dirichlet-Neumann operator defined as in \eqref{def-omegas}-\eqref{def-DN}. 

\begin{lemma}\label{lem-DN} { For $\omega \in H^{1/2}(\partial \Omega)$, let $DN \omega $ be the Dirichlet-Neumann operator defined as in \eqref{def-omegas}-\eqref{def-DN}.} In the scaled variables, there holds
\begin{equation}\label{comp-DN}
\widetilde{DN} \wtd\omega = |\partial_{\wtd \theta}| \wtd\omega + \widetilde B\wtd\omega
\end{equation}
for some linear bounded operator $\widetilde B$ from $L^2(\partial \Omega)$ to itself: namely, 
$$\| \wtd B \wtd\omega\|_{L^2(\partial \Omega)}
\le C_0  \| \wtd \omega\|_{L^2(\partial \Omega)} $$
for some positive constant $C_0$.
\end{lemma}
\begin{proof} { Let $\phi^b$ be the cutoff function defined as in \eqref{def-phib}, and set $\omega^{*b} = \phi^b \omega^*$, where $\phi^*$ solves \eqref{def-omegas}. It follows that   
\begin{equation}\label{def-omegas-b}\left\{
\begin{aligned}\Delta \omega^{*b} &= (\triangle \phi^b)\w^*-2\nabla\phi^b\cdot\nabla\w^*, \qquad \mbox{in}\quad \Omega
\\
{\omega^{*b}} &=\omega, \qquad \mbox{on }\quad \partial \Omega.
\end{aligned}
\right.
\end{equation}
}
Since $\phi^p$ vanishes away from the boundary, we can work in the scaled variables, which reads $\widetilde{DN} \wtd \w=-\pt_{\wtd z}\wtd \w^*_{\vert_{\wtd z=0}}$. Recalling \eqref{Lap-scaled}, the scaled function $\wtd \w^*(\wtd t, \wtd z, \wtd \theta)$ of $\omega^{*b}$ solves
\[
\triangle_{\wtd z,\wtd\theta}\wtd \w^*=-\lambda^2 \wtd R_{\Delta}\wtd \w^* + \lambda^{-2}[ (\triangle \phi^b)\w^*-2\nabla\phi^b\cdot\nabla\w^*],\qquad \wtd \w^*|_{\wtd z=0}=\wtd \w|_{\wtd z=0}\,,
\]
on $ \RR_+ \times \TT$, which can be solved explicitly. Indeed, let $\wtd\omega_\alpha$ be the Fourier coefficient of $\wtd \omega(\wtd z,\wtd\theta)$ in variable $\wtd\theta$. { Note that $\wtd \omega_\alpha$ vanishes for $\alpha =0$, and thus we focus on the case when $\alpha\not =0$.} Let  $K_\alpha(\wtd y, \wtd z)  = \frac{1}{2|\alpha|} (e^{-|\alpha(\wtd y - \wtd z)|} - e^{-|\alpha(\wtd y + \wtd z)|}) $ be the Green function of the Laplacian $\partial_{\wtd z}^2 - \alpha^2$ with the Dirichlet boundary condition. It follows that
\begin{equation}\label{def-wstar}
\begin{aligned}
 \wtd \omega^*_\alpha (\wtd z) &= e^{- |\alpha| \wtd z } \wtd \omega_\alpha (0)+ \lambda^2 \int_0^\infty K_\alpha(\wtd y, \wtd z)  (\wtd R_{\Delta}\wtd \w^*)_\alpha (\wtd y) \; d\wtd y
 \\ & \qquad {+ \lambda^{-2}  \int_0^\infty K_\alpha(\wtd y, \wtd z) \Big[(\triangle \phi^b)\w^*-2\nabla\phi^b\cdot\nabla\w^*\Big]_\alpha (\wtd y) \; d\wtd y}
\end{aligned}\end{equation}
for $\wtd z\ge 0$.
The Dirichlet-Neumann operator is thus computed by
$$
\begin{aligned}
(\widetilde{DN} \wtd \w)_\alpha &=-\pt_{\wtd z}\wtd \w^*_\alpha(0)
\\& = |\alpha| \wtd \omega_\alpha (0)+ \int_0^\infty e^{-|\alpha| \wtd y} \Big[  \lambda^2 (\wtd R_{\Delta}\wtd \w^*)_\alpha +  \lambda^{-2}( (\triangle \phi^b)\w^*-2\nabla\phi^b\cdot\nabla\w^*)_\alpha\Big](\wtd y)  \; d\wtd y.
\end{aligned}$$
The decomposition \eqref{comp-DN} thus follows, upon defining $\widetilde B$ as the integral term
\begin{equation}\label{def-BDN} (\wtd B\wtd \omega)_\alpha : =   \int_0^\infty e^{-|\alpha| \wtd y} \Big[  \lambda^2 (\wtd R_{\Delta}\wtd \w^*)_\alpha +  \lambda^{-2}( (\triangle \phi^b)\w^*-2\nabla\phi^b\cdot\nabla\w^*)_\alpha\Big](\wtd y)  \; d\wtd y\,, \end{equation}
for each Fourier variable $\alpha \in \ZZ$.
It remains to prove the boundedness of $\widetilde B$. Note that by definition, the last two terms are defined on the region $\wtd y \ge \delta_0 + \rho_0$ where the cutoff function $\phi^b=1$. Therefore, 
$$\Big|\int_0^\infty e^{-|\alpha| \wtd y} ( (\triangle \phi^b)\w^*-2\nabla\phi^b\cdot\nabla\w^*)_\alpha(\wtd y)  \; d\wtd y\Big| \lesssim  \| \omega^\star\|_{H^1(\lambda d(x,\partial \Omega) \ge \delta_0 + \rho_0)} .$$
It remains to bound the first integral term in \eqref{def-BDN}. In view of \eqref{def-RDelta}, we write
$$
\begin{aligned}
\wtd R_\Delta \wtd \omega^* &=  \del^2_{\wtd \theta} [ \wtd m \wtd \omega^*] - \del_{\wtd \theta} \Big[ 2\del_{\wtd \theta}  \wtd m \wtd \omega^* + \frac{\wtd z\wtd \gamma'}{(1+\lambda^2\wtd z\wtd \gamma)^3} \wtd \omega^*\Big ] +  \del_{\wtd z}\Big( \frac{\wtd \gamma}{1+\lambda^2\wtd z\wtd \gamma} \wtd \omega^*\Big)
\\&\quad + \Big[  (\del^2_{\wtd \theta}  \wtd m) -   \del_{\wtd z}\Big( \frac{\wtd \gamma}{1+\lambda^2\wtd z\wtd \gamma}\Big) +  \del_{\wtd \theta}\Big(\frac{\wtd z\wtd \gamma'}{(1+\lambda^2\wtd z\wtd \gamma)^3} \Big) \Big] \wtd \omega^* ,
\end{aligned}$$
noting the coefficients are analytic near the boundary. We note in particular that there is no growth in large $\wtd z$: for instance, $m(\wtd z,\wtd\theta) \lesssim \lambda^{-2}$ uniformly in large $\wtd z$. In addition, we note that $\wtd m = \wtd z \wtd m_1$ for some bounded function $\wtd m_1$. Thus, using the fact that $|\alpha| \wtd y e^{-\frac12 |\alpha| \wtd y} \lesssim 1$, the second-order derivative term $  \del^2_{\wtd \theta} [ \wtd m \wtd \omega^*] $ thus can be treated as the first order derivative term. Precisely, we can treat the first integral in \eqref{def-BDN} systematically as follows: for some smooth and bounded coefficients $b(\wtd z, \wtd \theta)$, 
$$   \lambda^2 \int_0^\infty e^{-\frac12 |\alpha| \wtd y} |(\alpha, \partial_{\wtd y}) (b \wtd \omega^*)_\alpha| (\wtd y) \; d\wtd y \lesssim    \lambda^2 |\alpha|^{-1/2}\| (\alpha, \partial_{\wtd y}) (b \wtd \omega^*)_\alpha \|_{L^2_{\wtd y}} .$$
This yields
$$
\begin{aligned}
|(\wtd B \wtd \omega)_\alpha|
&\lesssim    \lambda^2 |\alpha|^{-1/2}\| (\alpha, \partial_{\wtd y}) (b \wtd \omega^*)_\alpha \|_{L^2_{\wtd y}} +  \| \omega^\star\|_{H^1(\lambda d(x,\partial \Omega) \ge \delta_0 + \rho_0)} \,.
\end{aligned}
$$
Taking $L^2_\alpha$, we thus obtain
\begin{equation}\label{bd-Bw1}
\begin{aligned}
\sum_\alpha|(\wtd B \wtd \omega)_\alpha|^2
&\lesssim    \lambda^2 \sum_\alpha |\alpha|^{-1}\| (\alpha, \partial_{\wtd y}) \wtd \omega^*_\alpha \|_{L^2_{\wtd y}}^2,
\end{aligned}\end{equation}
upon noting that the coefficients $b(\wtd z, \wtd \theta)$, which in particular have $\| b_\alpha(\wtd z)\|_{L^1_\alpha L^\infty_{\wtd z}} <\infty$. It remains to bound the right-hand side of \eqref{bd-Bw1}.
Directly from \eqref{def-wstar}, we compute
$$
\begin{aligned}
| (\alpha, \partial_{\wtd z}) \wtd \omega^*_\alpha(\wtd z)|
&\lesssim |\alpha|e^{- |\alpha| \wtd z } |\wtd \omega_\alpha (0)|+ \lambda^2 \int_0^\infty e^{-|\alpha(\wtd z - \wtd z')| } | (\wtd R_{\Delta}\wtd \w^*)_\alpha (\wtd z')| \; d\wtd z'  + |\alpha|^{-1/2}  \| \omega^\star\|_{H^1(\lambda d(x,\partial \Omega) \ge \delta_0 + \rho_0)}.
\end{aligned}
$$
Therefore, together with the standard Hausdorff-Young's inequality, we bound
$$
\begin{aligned}
 \| (\alpha, \partial_{\wtd z}) \wtd \omega^*_\alpha\|_{L^2_{\wtd z}}
&\lesssim |\alpha|^{1/2}|\wtd \omega_\alpha (0)|+ \lambda^2 |\alpha|^{-1}  \| (\wtd R_{\Delta}\wtd \w^*)_\alpha\|_{L^2_{\wtd z}} + |\alpha|^{-1/2}  \| \omega^\star\|_{H^1(\lambda d(x,\partial \Omega) \ge \delta_0 + \rho_0)}
\end{aligned}
$$
which yields
$$
\begin{aligned}
\sum_\alpha |\alpha|^{-1}\| (\alpha, \partial_{\wtd z}) \wtd \omega^*_\alpha\|_{L^2_{\wtd z}}^2
&\lesssim \sum_\alpha|\wtd \omega_\alpha (0)|^2+ \lambda^2 \sum_\alpha |\alpha|^{-3}  \| (\wtd R_{\Delta}\wtd \w^*)_\alpha\|_{L^2_{\wtd z}} ^2 +  \| \omega^\star\|_{H^1(\lambda d(x,\partial \Omega) \ge \delta_0 + \rho_0)}^2
\\
&\lesssim \sum_\alpha|\wtd \omega_\alpha (0)|^2+ \lambda^2 \sum_\alpha |\alpha|^{-1} \| (\alpha, \partial_{\wtd z}) \wtd \omega^*_\alpha\|_{L^2_{\wtd z}}^2 
\\&\quad +
\sum_\alpha |\alpha|^{-1}\| (\alpha, \partial_{\wtd z}) \wtd \omega^*_\alpha\|^2_{L^2_{\{{\wtd z \ge \delta_0 + \rho_0}\}}} +  \| \omega^\star\|_{H^1(\lambda d(x,\partial \Omega) \ge \delta_0 + \rho_0)}.
\end{aligned}
$$
Taking $\lambda$ sufficiently small so that the second term on the right can be absorbed into the left. On the other hand, using the standard elliptic theory, the last term is bounded by
$$\sum_\alpha |\alpha|^{-1}\| (\alpha, \partial_{\wtd z}) \wtd \omega^\star_\alpha\|^2_{L^2_{\{{\wtd z \ge \delta_0 + \rho_0}\}}} \lesssim \| \omega^\star\|^2_{H^1(\lambda d(x,\partial \Omega) \ge \delta_0)} \lesssim \| \omega\|_{L^2(\partial\Omega)}^2.$$
Putting these back into \eqref{bd-Bw1}, we obtain the lemma.
\end{proof}

\section{Near boundary analytic spaces}\label{sec-norm}

In this section, we introduce the near boundary analytic norm used to control the vorticity that is analytic near the boundary, but however only has Sobolev regularity away from the boundary. We then derive sufficient elliptic estimates, bilinear estimates, as well as the semigroup estimates in these analytic spaces.

\subsection{Analytic norms}

Let $\delta>0$ be small and so that Proposition \ref{prop-deltad} applies for $\bar V_{\delta}=\{ d(x,\partial\Omega)\le \delta\}$. In particular, $\delta$ is small so that the statement of \ref{prop-deltad} still holds for $V_{2\delta}$. Now for any constant $\lambda \in (0,1)$, we have
\[
\lambda d(x,\partial\Omega) \le \lambda \delta
\]
for all $x\in \bar V_{\delta}$. Let $\delta_0=\lambda \delta$, which will the size of the analytic domain for our solution near the boundary. We fix $\rho_0\in(0,1/10)$, and assume that $\rho\in(0,\rho_0)$.
Then
  \begin{equation}
   \Omega_\rho=\{\wtd z\in \CC: 0 \leq \Re \wtd z \leq \delta_0, |\Im \wtd z|\le \rho \Re \wtd z\} \cup \{\wtd z \in \CC: \delta_0 \leq \Re \wtd z \leq \delta_0+\rho, |\Im \wtd z|\le \delta_0 + \rho -  \Re \wtd z  \}
  \end{equation}
denotes the complex domain for functions of the $\wtd z$ variable.
We note that the domain $\Omega_\rho$ only contains $\wtd z$ with $0\le \Re \wtd z\le \delta_0+\rho$.
For a complex valued function  $f$ defined on $\Omega_\rho$, let
\[
\|f\|_{L ^1_\rho}=\sup_{0 \le \eta <\rho}\| f \|_{L^1(\pt\Omega_\eta)} , \qquad \|f\|_{L ^\infty_\rho}=\sup_{0 \le \eta <\rho}\| f \|_{L^\infty(\pt\Omega_\eta)}
\]
where the integration is taken over the two directed paths along the boundary of the
domain $\Omega_{\eta}$.
Now for an analytic function $f(\wtd \theta,\wtd z)$ defined on $(\wtd \theta,\wtd z)\in \mathbb{T}\times \Omega_\rho$, we define
\beq\label{normY-mu}
\begin{aligned}
\|f\|_{\mathcal L^1_\rho} &=\sum_{\alpha\in \mathbb{Z}}\|e^{\eps_0(\delta_0+\rho- \Re \wtd z)|\alpha|}f_\alpha\|_{L^1_\rho} ,
\\
 \|f\|_{\mathcal L^\infty_\rho} &=\sum_{\alpha\in \mathbb{Z}}\|e^{\eps_0(\delta_0+\rho- \Re \wtd z)|\alpha|}f_\alpha\|_{L^\infty_\rho} ,
 \end{aligned}
\eeq
where $f_\alpha$ denotes the Fourier transform of $f$ with respect to variable $\wtd\theta$. The function spaces $\cL^1_\rho$ and $\cL_\rho^\infty$ are to control the scaled vorticity and velocity, respectively. We stress that the analyticity weight vanishes on $\Re \wtd z \ge \delta_0+\rho$. For convenience, we also introduce the following analytic norms
\begin{equation}\label{def-W1k}\| f\|_{\cW_\rho^{k,p}} =  \sum_{i+j\le k}\|\pt_{\wtd \theta}^i (\wtd z\pt_{\wtd z})^j f\|_{\cL^p_\rho}
\end{equation}
for $k\ge 0$ and $p=1, \infty$.
We observe the following simple algebra.

\begin{lemma} There hold
\begin{equation}\label{Y-algebra} \| fg \|_{\cL^1_\rho} \le \| f\|_{\cL^\infty_\rho } \| g\|_{\cL^1_\rho}
\end{equation}
and for any $0<\rho'<\rho$,
\begin{equation}\label{Y-derivative}
\| \partial_{\wtd \theta} f \|_{\cL^1_{\rho'}} + \| \wtd z \partial_{\wtd z} f\|_{\cL^1_{\rho'}} \lesssim \frac{1}{\rho - \rho'} \| f\|_{\cL^1_\rho}.
\end{equation}

\end{lemma}

\begin{proof} By definition, we compute
$$
\begin{aligned}
e^{\eps_0(\delta_0+\rho- \Re \wtd z)|\alpha|}|(fg)_\alpha (\wtd z)| &\le \sum_{\alpha'} |f_{\alpha-{\alpha'}}(\wtd z) g_{\alpha'} (\wtd z)| e^{\eps_0(\delta_0+\rho- \Re \wtd z)|\alpha|}
\\
&\le \sum_{\alpha'} | e^{\eps_0(\delta_0+\rho- \Re \wtd z)|\alpha-{\alpha'}|}f_{\alpha-{\alpha'}}(\wtd z) e^{\eps_0(\delta_0+\rho- \Re \wtd z)|{\alpha'}|}g_{\alpha'} (\wtd z)|
\end{aligned}$$
which gives
$$
\begin{aligned}
\| e^{\eps_0(\delta_0+\rho- \Re \wtd z)|\alpha|}(fg)_\alpha (\wtd z)\|_{\mathcal{L}_\rho^1}
&\le \sum_{\alpha'} \| e^{\eps_0(\delta_0+\rho- \Re \wtd z)|\alpha-{\alpha'}|}f_{\alpha-{\alpha'}}\|_{\mathcal L^\infty_\rho} \| e^{\eps_0(\delta_0+\rho- \Re \wtd z)|{\alpha'}|}g_{\alpha'} \|_{\mathcal L^1_\rho} .
\end{aligned}$$
The estimate \eqref{Y-algebra} follows from taking the summation in $\alpha$ over $\ZZ$. The stated bounds on derivatives are classical (e.g., \cite{CS98, NN18}), making use of the fact that $(\rho-\rho')|\alpha| e^{(\rho'-\rho)|\alpha|}$ is bounded.
\end{proof}

\subsection{Elliptic estimates in the half-plane}

In this section, we derive some basic elliptic estimates in the analytic spaces $\cW^{k,p}_\rho$. Precisely, we consider
\begin{equation}\label{elliptic}
\left \{
\begin{aligned}
\Delta_{z,\theta} \phi &= f, \qquad \mbox{in }\quad \RR_+ \times \TT
\\
\phi_{\vert_{z=0}} &=0
\end{aligned}\right.
\end{equation}
in which we drop titles for sake of presentation. The $\cW^{k,p}_\rho$ analytic norm is defined on $\Re z\le \delta_0 + \rho$ as introduced in the previous section.
We obtain the following proposition.

\begin{prop} \label{prop-elliptic}
Let $\phi$ be the solution of \eqref{elliptic}. Then, the velocity field $u= \nabla^\perp \phi$ satisfies
\beq \label{elliptic1}
\begin{aligned}
\|u\|_{\cW^{k,\infty}_\rho } & \lesssim \| f\|_{\cW^{k,1}_\rho} + \|f\|_{H^{k+1}(\{z\ge \delta_0+\rho\})}
\\
\|(\frac{1}{ z} \partial_{ \theta} \phi) \|_{\cW^{k,\infty}_\rho }
&\lesssim  \| f\|_{\cW^{k,1}_\rho} + \|\partial_{ \theta}f\|_{\cW^{k,1}_\rho} + \|f\|_{H^{k+1}(\{z\ge \delta_0 + \rho\})}
\\
\| \nabla_{z,\theta}u\|_{\cW^{k,\infty}_\rho } & \lesssim \| f\|_{\cW^{k,\infty}_\rho} + \|f\|_{H^{k+2}(\{z\ge \delta_0+\rho\})}
\end{aligned}
\eeq
for $k\ge 0$.
\end{prop}

\begin{proof}
The elliptic problem \eqref{elliptic} can be solved explicitly in Fourier space. Indeed, taking the Fourier transform in $\theta$, we get the elliptic equation
$$(\partial_z^2 -\alpha^2) \phi_\alpha = f_\alpha$$ for the Fourier transform $\phi_\alpha$.
We focus on the case $\alpha>0$; the other case is similar. The solution is given by
$$ \phi_\alpha( z)=\int_{0}^{ z} K_-(y, z) f_\alpha(y) dy + \int_{ z}^\infty K_+(y, z) f_\alpha(y)dy$$
with the Green function defined by  $$K_\pm(y,z) = -\frac{1}{2\alpha} \Big( e^{\pm\alpha (z-y) } - e^{-\alpha (y+z)} \Big).
$$
This expression may be extended to complex values of $z$. Indeed, for $z\in \Omega_\sigma$, there is a positive $\theta$ so that $z\in \partial\Omega_\theta$. We then write  $\partial\Omega_\theta = \gamma_-(z) \cup \gamma_+(z) $, consisting of complex numbers $y \in \partial\Omega_\theta$ so that $\Re y < \Re z$ and $\Re y > \Re z$, respectively. Then, the integral is taken over $\gamma_-( z)$ and $\gamma_+( z)$, respectively. We note in particular that for $y \in \gamma_\pm(z)$, there hold the same bounds on the Green function
$$ |K_\pm(y,z)|\le \alpha^{-1} e^{-\alpha |y-z|} .$$
This proves that
\begin{equation}\label{bd-phiz0}
\begin{aligned}
 | \phi_\alpha( z)| &\le \int_{\partial \Omega_\theta} \alpha^{-1} e^{-\alpha |y- z|}|f_\alpha(y)| |dy|.
\end{aligned}\end{equation}
By definition of $\cL^1_\rho$ norm, we only need to consider the case when $0\le \Re  z \le \delta_0 + \rho$. Now, for $0\le \Re y \le  \delta_0 + \rho$, we bound
$$  e^{-\alpha |\Re y- \Re  z|} e^{-\eps_0(\delta_0+\rho- \Re y)\alpha} \le e^{-\eps_0(\delta_0+\rho- \Re  z)\alpha} e^{-(1-\epsilon_0)\alpha |\Re y -\Re  z|}$$
noting $\epsilon_0 \le 1/2$. On the other hand, for $\Re y \ge \delta_0 + \rho$ (recalling $\delta_0 + \rho \ge \Re  z$), we bound
$$  e^{-\alpha |\Re y- \Re  z|} \le e^{- \epsilon_0(\delta_0+\rho- \Re  z) \alpha }  e^{- (1-\epsilon_0)\alpha |\Re y- \Re  z|}.$$
Therefore, we bound
$$
\begin{aligned}
 \int_{\Re y \le \delta_0 + \rho} \alpha^{-1} e^{-\alpha |y- z|}  | f_\alpha(y)|  |dy| &\lesssim \alpha^{-1} e^{-\eps_0(\delta_0+\rho- \Re  z)\alpha} \| e^{\eps_0(\delta_0+\rho- \Re y)\alpha}f_\alpha\|_{L^1_\rho},
 \\
 \int_{\Re y \ge \delta_0 + \rho} \alpha^{-1} e^{-\alpha |y- z|}  | f_\alpha(y)|  |dy| &\lesssim \alpha^{-3/2} e^{-\eps_0(\delta_0+\rho- \Re  z)\alpha} \| f_\alpha\|_{L^2(y \ge \delta_0 + \rho)}.
 \end{aligned}$$
Similarly, we also have
$$
\begin{aligned}
 \int_{\Re y \le \delta_0 + \rho} \alpha^{-1} e^{-\alpha |y- z|}  | f_\alpha(y)|  |dy| &\lesssim \alpha^{-2} e^{-\eps_0(\delta_0+\rho- \Re  z)\alpha} \| e^{\eps_0(\delta_0+\rho- \Re y)\alpha}f_\alpha\|_{L^\infty_\rho},
 \end{aligned}$$
which gains an extra factor of $\alpha$. This proves
$$
\begin{aligned}
 \| e^{\eps_0(\delta_0+\rho- \Re  z)\alpha}(\alpha,\partial_z) \phi_\alpha\|_{L^\infty_\rho} &\le \| e^{\eps_0(\delta_0+\rho- \Re y)\alpha}f_\alpha\|_{L^1_\rho} +  \alpha^{-1/2}\| f_\alpha\|_{L^2(y \ge \delta_0 + \rho)}
 \\
  \| e^{\eps_0(\delta_0+\rho- \Re  z)\alpha} (\alpha,\partial_z)^2 \phi_\alpha\|_{L^\infty_\rho} &\le \| e^{\eps_0(\delta_0+\rho- \Re y)\alpha}f_\alpha\|_{L^\infty_\rho} +  \alpha^{1/2}\| f_\alpha\|_{L^2(y \ge \delta_0 + \rho)}
. \end{aligned}$$
Taking the summation in $\alpha \in \ZZ$ yields the first and last estimates in \eqref{elliptic1} for $k=0$. For $k\ge 0$, the estimates follow similarly. For the estimates involving the weight $z^{-1}$, we use the fact that the Green function vanishes on the boundary $z=0$, and so $ |G_\pm(y,z)|\le z e^{-\alpha |y-z|} .$
\end{proof}

\subsection{Biot-Savart law in $\Omega$}\label{sec-elliptic}

In this section, we bound the velocity through the Biot-Savart law: namely, $u = \nabla^\perp \phi$, where
\begin{equation}\label{BSlaw}
\left \{
\begin{aligned}
\Delta \phi &= \omega , \qquad \mbox{in }\quad \Omega
\\
\phi &=0, \qquad \mbox{on}\quad \partial \Omega .
\end{aligned}\right.
\end{equation}
{Without loss of generality, we will work with the cut-off vorticity $\omega^b$ (see Section \ref{sec-nonlinear}) near the boundary where the rescaled coordinates introduced in Section \ref{sec-bdryvorticity} apply.} 
We obtain the following proposition.

\begin{prop} \label{inverseLaplace}
Let $\phi$ be the solution of \eqref{BSlaw}. Then, the velocity field $u= \nabla^\perp \phi$ satisfies
\beq \label{laplace-5}
\begin{aligned}
\|u\|_{\cW^{k,\infty}_\rho } & \lesssim \|\omega\|_{\cW^{k,1}_\rho} + \|\omega\|_{H^{k+1}(\{\lambda d(x,\partial\Omega)\ge \delta_0/2\})}
\\
\|(\frac{1}{\wtd z} \partial_{\wtd \theta} \phi) \|_{\cW^{k,\infty}_\rho }
&\lesssim  \|\omega\|_{\cW^{k,1}_\rho} + \|\partial_{\wtd \theta}\omega\|_{\cW^{k,1}_\rho} + \|\omega\|_{H^{k+1}(\{\lambda d(x,\partial\Omega)\ge \delta_0/2\})}
\end{aligned}
\eeq
for $k\ge 0$.
\end{prop}

\begin{proof}
Using \eqref{Lap-scaled} and \eqref{BSlaw}, the scaled stream function $\wtd \phi(\wtd t, \wtd z, \wtd \theta)$ solves
\[
\Delta_{\wtd z, \wtd \theta}\wtd \phi = \lambda^{-2} \wtd \omega -\lambda^2 \wtd R_{\Delta}\wtd \phi  ,\qquad {\wtd \phi}{}_{\vert_{\wtd z=0}}= 0
\]
on $\TT \times \RR_+$, and so the elliptic theory, Proposition \ref{prop-elliptic}, developed in the previous section can be applied, yielding
\begin{equation}\label{BSbd-u1}
\begin{aligned}
 \|u\|_{\cW^{k,\infty}_\rho }  &\lesssim \|\omega\|_{\cW^{k,1}_\rho} + \|\omega\|_{H^{k+1}(\{\lambda d(x,\partial\Omega)\ge \delta_0+\rho\})}
 \\&\quad + \lambda^2 \| \partial_{\wtd \theta}^{-1} \wtd R_{\Delta}\wtd \phi \|_{\cW^{k,\infty}_\rho} + \lambda^2 \| \wtd R_{\Delta}\wtd \phi \|_{H^{k+1}(\{\wtd z\ge \delta_0+\rho\})}  .
 \end{aligned}\end{equation}
It thus remains to bound $ \wtd R_{\Delta}\wtd \phi$. Recall from \eqref{def-RDelta} that
$$
\wtd R_\Delta  = \wtd m(\wtd z,\wtd \theta) \del^2_{\wtd \theta} +  \frac{\wtd \gamma}{1+\lambda^2\wtd z\wtd \gamma} \del_{\wtd z} -\frac{\wtd z\wtd \gamma'}{(1+\lambda^2\wtd z\wtd \gamma)^3} \del_{\wtd \theta}  ,\qquad \wtd m(\wtd z,\wtd\theta)= -
 \frac{2\wtd z\wtd \gamma+\lambda^2(\wtd z\wtd \gamma)^2}{(1+\lambda^2\wtd z\wtd \gamma)^2} .
$$
Thanks to the analyticity of the boundary, the coefficients are clearly bounded in $\cW^{k,\infty}_\rho$. Therefore, using a similar algebra as in \eqref{Y-algebra}, we bound
\begin{equation}\label{analytic-Rdelta} \lambda^2\| \partial_{\wtd \theta}^{-1}\wtd R_\Delta \wtd \phi \|_{\cW^{k,\infty}_\rho} \lesssim  \lambda^2 \| \partial_{\wtd \theta}\wtd \phi \|_{\cW^{k,\infty}_\rho} + \lambda^2 \| \partial_{\wtd z}\wtd \phi \|_{\cW^{k,\infty}_\rho}\end{equation}
That is, this term can be absorbed into the left hand side of \eqref{BSbd-u1}, upon taking $\lambda$ sufficiently small. As for the last term in \eqref{BSbd-u1}, we note that for large $\wtd z$, $|\wtd m(\wtd z,\wtd\theta)| \lesssim \lambda^{-2}$, which in particular proves that there is no growth in $\wtd z$. This gives
\begin{equation}\label{Sobolev-Rdelta}
\begin{aligned}
 \lambda^2 \| \wtd R_{\Delta}\wtd \phi \|_{H^{k+1}(\{\wtd z\ge \delta_0+\rho\})}
 &\lesssim  \| \phi \|_{H^{k+3}(\{\lambda d(x,\partial\Omega)\ge \delta_0+\rho\})}
 \\
 &\lesssim  \lambda^2 \| \wtd \phi \|_{\cW^{k,\infty}_\rho} + \| \omega \|_{H^{k+1}(\{\lambda d(x,\partial\Omega)\ge \delta_0 /2\})}
, \end{aligned}
 \end{equation}
 in which the last estimate follows from the standard elliptic theory in Sobolev spaces.
 The proposition follows.
 \end{proof}

\subsection{Bilinear estimates}

In this section, we show that the Sobolev-analytic norm is well adapted to treat the nonlinear $u \cdot \nabla \omega$. We have the following lemma.

\begin{lemma}\label{lem-bilinear}
For any $\omega$ and $\omega'$,
denoting by $u$ the velocity related to $\omega$, we have
$$
\begin{aligned}
 \| u\cdot \nabla \omega' \|_{\cL^1_\rho} &\le C \Big( \| \omega\|_{\cL^1_\rho} + \|\omega\|_{H^{1}(\{\lambda d(x,\partial\Omega)\ge \delta_0\})} \Big)\| \partial_{\wtd \theta}\omega'\|_{\cL^1_\rho}
 \\& \quad + C \Big ( \|\omega\|_{\cL^1_\rho} + \|\partial_{\wtd \theta}\omega\|_{\cL^1_\rho} + \|\omega\|_{H^{1}(\{\lambda d(x,\partial\Omega)\ge \delta_0\})} \Big) \| \wtd z \partial_{\wtd z}\omega'\|_{\cL^1_\rho} .
 \end{aligned}$$
\end{lemma}
\begin{proof} By definition, the $\cL^1_\rho$ norm is defined near the boundary $\{\lambda d(x,\partial\Omega) \le \delta_0 + \rho\}$, on which we can write
\[
u\cdot\nabla \w'=\frac{1}{ 1 + z\gamma(\theta) }\pt_\theta\phi\pt_z\w' - \frac{1}{( 1 + z\gamma(\theta) )^2}\pt_z\phi \pt_\theta\w'
\]
with $\Delta \phi = \omega$. In the rescaled variable $(\wtd z,\wtd \theta)$, we get
\[
u\cdot\nabla \w'= \frac{\lambda^2}{1+\lambda^2\wtd z\wtd \gamma(\wtd \theta)} (\pt_{\wtd \theta}\wtd \phi)(\pt_{\wtd z}\wtd \w ') -\frac{\lambda^2}{(1+\lambda^2\wtd z\wtd \gamma(\wtd \theta))^2} (\pt_{\wtd z}\wtd \phi)(\pt_{\wtd \theta}\wtd \w ')
\]
Note that thanks to the analyticity of $\partial \Omega$, the coefficient $(1+\lambda^2\wtd z\wtd \gamma(\wtd \theta))^{-1}$ is bounded in $\cL^\infty_\rho $. Using \eqref{Y-algebra} and Proposition \ref{inverseLaplace}, we bound

$$
\begin{aligned}
 \| (\pt_{\wtd z}\wtd \phi)(\pt_{\wtd \theta}\wtd \w ') \|_{\cL^1_\rho} &\lesssim \| \pt_{\wtd z}\wtd \phi \|_{\cL_\rho^\infty} \|\pt_{\wtd \theta}\wtd \w ' \|_{\cL^1_\rho}
 \\&\lesssim
 \Big( \| \omega\|_{\cL^1_\rho} + \|\omega\|_{H^{1}(\{\lambda d(x,\partial\Omega)\ge \delta_0\})} \Big)\| \partial_{\wtd \theta}\omega'\|_{\cL^1_\rho}
\\ \| (\pt_{\wtd \theta}\wtd \phi)(\pt_{\wtd z}\wtd \w ') \|_{\cL^1_\rho} &\lesssim \| \frac{1}{\wtd z}\pt_{\wtd \theta}\wtd \phi \|_{\cL_\rho^\infty} \| \wtd z\pt_{\wtd z}\wtd \w ' \|_{\cL^1_\rho}
 \\&\lesssim \Big( \|\omega\|_{\cL^1_\rho} + \|\partial_{\wtd \theta}\omega\|_{\cL^1_\rho} + \|\omega\|_{H^{1}(\{\lambda d(x,\partial\Omega)\ge \delta_0\})} \Big) \| \wtd z \partial_{\wtd z}\omega'\|_{\cL^1_\rho}
 \end{aligned}$$
giving the lemma.
\end{proof}

\subsection{Semigroup estimates in the half-plane}\label{sec-Stokes}

In this section, we give bounds on the Stokes semigroup $e^{\nu t S}$ in the analytic spaces $\cW^{k,1}_\rho$ on the half-plane $\RR_+\times \TT$. We also denote by $\Gamma(\nu t) = e^{\nu  t S}(\mathcal{H}^1_{\{\wtd z=0\} \times \TT})$ the trace of the semigroup on the boundary, with $\mathcal{H}^1_{\{\wtd z=0\}\times \TT}$ being the one-dimensional Hausdorff measure restricted on the boundary. The results in this section are an easy adaptation from those obtained in \cite{NN18}, where the analytic spaces contained no cutoff in $z$.
Precisely, we consider
\begin{equation}\label{lin-Stokes}
\begin{aligned}
(\pt_t-\nu \triangle_{z,\theta} ) \w &=0
\\
\nu (\partial_z + |\partial_{\theta}|)\omega_{\vert_{z=0}} &=0
\end{aligned}\end{equation}
on $\RR_+\times \TT$ (where we drop titles for sake of presentation). We obtain the following proposition.

\begin{proposition}\label{prop-Stokes}

Let $e^{\nu t S}$ be the semigroup of the linear Stokes problem \eqref{lin-Stokes}, and let $\Gamma(\nu t)g$ be its trace on the boundary.
 Then, for any $t\ge 0$, $\rho>0$, and $k\ge 0$, there hold
\begin{equation}\label{exp-SW}
\begin{aligned}
\| e^{\nu t S} f\|_{\cW^{k,1}_\rho} &\le C_0 \| f\|_{\cW^{k,1}_\rho} +  \| z f\|_{H^{k+1}(z\ge \delta_0 + \rho)}
\\
\| \Gamma(\nu t)g\|_{\cW^{k,1}_\rho} &\le C_0 \sum_{\alpha\in \ZZ}| \alpha^k g_\alpha| e^{\epsilon_0 (\delta_0 + \rho)|\alpha|}
\end{aligned}\end{equation}
uniformly in the inviscid limit.
\end{proposition}

\begin{proof} The proof follows closely from that in \cite{NN18}. Indeed, taking the Fourier transform of the semigroup $e^{\nu t S}$ in variable $\theta$, we obtain
\begin{equation}\label{semigroup-exp1}
 (e^{\nu t S} f)_\alpha (z) =  \int_0^\infty  G_\alpha(t,y;z) f_{\alpha} (y) \; dy, \qquad  (\Gamma(\nu t)g)_\alpha(z) = G_\alpha(t, 0; z) g_\alpha,
\end{equation}
for each Fourier variable $\alpha \in \ZZ$, where $G_\alpha(t,y;z)$ is the corresponding Green function. We recall the following result of Proposition 3.3 from \cite{NN18} that
\begin{equation}\label{Stokes-tG}
G_\alpha(t,y;z) = H_\alpha(t,y;z) + R_\alpha (t,y;z),
\end{equation}
where
$$
\begin{aligned}
H_\alpha(t,y;z) & = \frac{1}{\sqrt{\nu t}} \Big( e^{-\frac{|y-z|^{2}}{4\nu t}} +  e^{-\frac{|y+z|^{2}}{4\nu t}} \Big) e^{-\alpha^{2}\nu t},
\\ | \partial_z^k R_\alpha (t,y;z)| &\lesssim \mu_f^{k+1} e^{-\theta_0\mu_f |y+z|} +  (\nu t)^{-\frac{k+1}{2}}e^{-\theta_0\frac{|y+z|^{2}}{\nu t}}  e^{-\frac18 \alpha^{2}\nu t} ,
\end{aligned}$$
for $y,z\ge 0$, $k\ge 0$, and for some $\theta_0>0$ and for $\mu_f = |\alpha| + \frac{1}{\sqrt \nu}$. In particular, $\|G_\alpha(t,y;\cdot)\|_{L^1_\rho} \lesssim 1$, for each fixed $y,t$.

Now, for $z, y \le \delta_0 + \rho$, we note that
\begin{equation}\label{exp-a}
\begin{aligned}
e^{- a|y\pm z|} e^{-\epsilon_0 (\delta_0 + \rho- y) |\alpha|}
&= e^{-a|y\pm z| + \epsilon_0 |\alpha| (y - z)} e^{-\epsilon_0 (\delta_0 + \rho- z) |\alpha|}
\\
&\le e^{- (a - \epsilon_0 |\alpha|) |y\pm z| } e^{-\epsilon_0 (\delta_0 + \rho- z) |\alpha|}
\end{aligned}
\end{equation}
for any real number $a$ and for $\epsilon_0$ sufficiently small. Taking $a = \frac12 \theta_0 \mu_f$, we have $a \ge \epsilon_0 |\alpha| $ and so
$$e^{-\theta_0\mu_f |y+z|}e^{-\epsilon_0 (\delta_0 + \rho- y) |\alpha|}   \le e^{-\epsilon_0 (\delta_0 + \rho- z) |\alpha|} e^{- \frac12\theta_0\mu_f |y+z|} $$
On the other hand, taking $a =\frac12\theta_0 \frac{|y\pm z|}{\nu t}$ in \eqref{exp-a},  we have either $a \ge \epsilon_0 |\alpha|$ or $\frac12 \theta_0 \alpha^2 \nu t \ge \epsilon_0 |\alpha| |y\pm z|$. Therefore, we have
$$e^{-\theta_0\frac{|y+z|^{2}}{\nu t}}  e^{-\theta_0 \alpha^{2}\nu t}  e^{-\epsilon_0 (\delta_0 + \rho- y) |\alpha|}   \le  e^{-\frac12 \theta_0\frac{|y+z|^{2}}{\nu t}}  e^{-\epsilon_0 (\delta_0 + \rho- z) |\alpha|} .$$
This proves that for $z\le \delta_0 + \rho$,
$$
\begin{aligned}
&e^{\epsilon_0 (\delta_0 + \rho- z) |\alpha|}  \int_0^{\delta_0 + \rho} |G_\alpha(t,y;z) f_{\alpha} (y)| \; dy
\\&\le \int_0^{\delta_0 + \rho} \Big[ (\nu t)^{-\frac{1}{2}}e^{- \frac12\theta_0\frac{|y\pm z|^{2}}{\nu t}} + \mu_f e^{- \frac12\theta_0\mu_f |y+z|} \Big]  |e^{\epsilon_0 (\delta_0 + \rho- y) |\alpha|}  f_{\alpha} (y) |\; dy .
\end{aligned}$$
Since the term in the bracket is bounded in $L^1_z$ norm, we have
$$
\begin{aligned}
\Big\|e^{\epsilon_0 (\delta_0 + \rho- z) |\alpha|}  \int_0^{\delta_0 + \rho} G_\alpha(t,y;z) f_{\alpha} (y) \; dy\Big\|_{\cL^1_\rho}
\lesssim \|e^{\epsilon_0 (\delta_0 + \rho- y) |\alpha|}  f_{\alpha} \|_{\cL^1_\rho} .
\end{aligned}$$
Taking the summation in $\alpha$ yields the stated bounds for this term.

Next, consider the case when $y\ge \delta_0 + \rho \ge z$. In this case, we simply use
$$e^{- \epsilon_0 |\alpha| |y-z|} \le e^{-\epsilon_0 |\alpha| (\delta_0 + \rho - z)},$$
giving the right analyticity weight in $z$. The control of the weight $e^{\epsilon_0 |\alpha| |y-z|} $ is done exactly as above, yielding
$$
\begin{aligned}
&e^{\epsilon_0 (\delta_0 + \rho- z) |\alpha|}  \int_{\delta_0 + \rho}^\infty |G_\alpha(t,y;z) f_{\alpha} (y)| \; dy
\\&\le \int_{\delta_0 + \rho}^\infty \Big[ (\nu t)^{-\frac{1}{2}}e^{- \frac12\theta_0\frac{|y\pm z|^{2}}{\nu t}} + \mu_f e^{- \frac12\theta_0\mu_f |y+z|} \Big]  | f_{\alpha} (y) |\; dy  .
\end{aligned}$$
Therefore,
$$
\begin{aligned}
\sum_\alpha \| e^{\epsilon_0 (\delta_0 + \rho- z) |\alpha|}  \int_{\delta_0 + \rho}^\infty |G_\alpha(t,y;z) f_{\alpha} (y)| \; dy \|_{\cL^1_\rho}
&\lesssim \sum_\alpha \| f_{\alpha} \|_{L^1(z\ge \delta_0 + \rho)}
\\
&\lesssim \| z f \|_{H^1(z\ge \delta_0 + \rho)}.
\end{aligned}$$

Similarly, from \eqref{semigroup-exp1}, the Fourier transform of the trace operator $\Gamma(\nu t)g$ is estimated by
$$
\begin{aligned}
|(\Gamma(\nu t)g)_\alpha(z)|
&\le |G_\alpha(t, 0; z) g_\alpha|
\\
& \le  \Big[ \mu_f e^{-\theta_0\mu_f |z|} +  (\nu t)^{-\frac12}e^{-\theta_0\frac{|z|^{2}}{\nu t}}  e^{-\frac18 \alpha^{2}\nu t}\Big] |g_\alpha|
\\
& \le  \Big[ \mu_f e^{- \frac12 \theta_0\mu_f |z|} +  (\nu t)^{-\frac12}e^{- \frac12 \theta_0\frac{|z|^{2}}{\nu t}} \Big] e^{-\epsilon_0 (\delta_0 + \rho - z)|\alpha|}|g_\alpha| e^{\epsilon_0 (\delta_0 + \rho)|\alpha|}
\end{aligned}
$$
in which the last inequality is a special case of the previous calculations for $y=0$ and $z\le \delta_0 + \rho$.
The bounds $\Gamma(\nu t)g$ are thus direct. Finally, the bounds on derivatives follow from the similar adaptation of derivatives bounds provided in \cite{NN18}. We skip repeating the details. \end{proof}

\subsection{Semigroup estimates near $\partial \Omega$}\label{sec-realStokes}

In this section, we provide bounds on the Stokes semigroup $e^{\nu t S}$, which will be used to estimate the vorticity $\omega^b$ (see Section \ref{sec-nonlinear}) near the boundary in the analytic spaces $\cW^{k,1}_\rho$. Precisely, we consider
\begin{equation}\label{real-Stokes}
\left\{
\begin{aligned}
 &\pt_t\w-\nu\triangle \w= 0\\
 &\nu(\pt_n+DN)\w_{\vert_{\partial \Omega}}=0
\end{aligned}
\right.
\end{equation}
in $\Omega$. We obtain the following proposition.

\begin{proposition}\label{prop-realStokes}
Let $e^{\nu t S}$ be the semigroup of the linear Stokes problem \eqref{real-Stokes}, and let $\Gamma(\nu t)$ be its trace on the boundary. Fix any finite time $T$.
Then, for sufficiently small $\lambda$, and for any $0\le t\le T$, $\rho>0$, and $k\ge 0$, there hold
\begin{equation}\label{exp-realSW}
\begin{aligned}
\| e^{\nu t S} f\|_{\cW^{k,1}_\rho} &\le C_0 \| f\|_{\cW^{k,1}_\rho} + \| f \|_{H^{k+1}(\lambda d(x,\partial\Omega) \ge \delta_0/2)}
\\
\| \Gamma(\nu t)g\|_{\cW^{k,1}_\rho} &\le C_0 \sum_{\alpha\in\ZZ} |\alpha^k g_\alpha|e^{\epsilon_0 (\delta_0 + \rho)|\alpha|}
\end{aligned}\end{equation}
uniformly in the inviscid limit.
\end{proposition}

\begin{proof} In the scaled variables, the Stokes problem for near boundary vorticity $\omega$ becomes
\[\begin{cases}
&(\pt_{\wtd t}-\nu\triangle_{\wtd z,\wtd \theta})\wtd \w=-\lambda^2\nu \wtd R_\Delta\wtd \w\\
&\nu(\pt_{\wtd z}+|\partial_{\wtd \theta}|)\wtd \w|_{\wtd s=0}= -\nu \wtd B \wtd \omega
\end{cases}
\]
where $\wtd R_\Delta$ and $\wtd B$ are defined as in \eqref{def-RDelta} and \eqref{def-BDN}. Using the Duhamel, the solution with initial data $\omega_0 $ can be written as
 \beq\label{Duh-w0}
\wtd \omega(\wtd t)=e^{\nu \wtd t S}\wtd \omega_{0} - \nu \lambda^2\int_{0}^{\wtd t}e^{\nu(\wtd t-\wtd t')S} \wtd R_\Delta\wtd \w(\wtd t') \; d\wtd t' - \nu \int_0^{\wtd t} \Gamma(\nu (\wtd t-\wtd t'))\wtd B \wtd \omega(\wtd t') \;d\wtd t'.
\eeq
We shall bound the integral terms on the right in term of the initial data. Recall from \eqref{def-RDelta} that
$$
\wtd R_\Delta  = \wtd m(\wtd z,\wtd \theta) \del^2_{\wtd \theta} +  \frac{\wtd \gamma}{1+\lambda^2\wtd z\wtd \gamma} \del_{\wtd z} -\frac{\wtd z\wtd \gamma'}{(1+\lambda^2\wtd z\wtd \gamma)^3} \del_{\wtd \theta}  ,\qquad \wtd m(\wtd z,\wtd\theta)= -
 \frac{2\wtd z\wtd \gamma+\lambda^2(\wtd z\wtd \gamma)^2}{(1+\lambda^2\wtd z\wtd \gamma)^2} .
$$
We rewrite the operator in the following form
$$
\begin{aligned}
\wtd R_\Delta \wtd \omega &=  \del^2_{\wtd \theta} [ \wtd m \wtd \omega] - \del_{\wtd \theta} \Big[ 2\del_{\wtd \theta}  \wtd m \wtd \omega + \frac{\wtd z\wtd \gamma'}{(1+\lambda^2\wtd z\wtd \gamma)^3} \wtd \omega\Big ] +  \del_{\wtd z}\Big( \frac{\wtd \gamma}{1+\lambda^2\wtd z\wtd \gamma} \wtd \omega\Big)
\\&\quad + \Big[  (\del^2_{\wtd \theta}  \wtd m) -   \del_{\wtd z}\Big( \frac{\wtd \gamma}{1+\lambda^2\wtd z\wtd \gamma}\Big) +  \del_{\wtd \theta}\Big(\frac{\wtd z\wtd \gamma'}{(1+\lambda^2\wtd z\wtd \gamma)^3} \Big) \Big] \wtd \omega  .
\end{aligned}$$
We now bound each term appearing in the Duhamel formula \eqref{Duh-w0}. Thanks to the analyticity of the boundary, the coefficients are bounded in $\cW^{k,\infty}_\rho$. Now, recall from \eqref{Stokes-tG} that the Green function has two components:
$$ e^{\nu \wtd t S} = e^{\nu \wtd t S_H} + e^{\nu \wtd t S_R} $$
which corresponds to the Green kernel $H_\alpha$ (i.e., the heat kernel) and the other from the stationary Stokes kernel $R_\alpha$.

We first claim that
\begin{equation}\label{claim-S1}
\Big \|  \nu \lambda^2\int_{0}^{\wtd t}e^{\nu(\wtd t-\wtd t')S_H} \wtd R_\Delta\wtd \w(\wtd t') \; d\wtd t'\Big\|_{\cW^{k,1}_\rho} \lesssim  \lambda^2 \sup_{0\le \wtd t' \le \wtd t} \| \omega \|_{\cW^{k,1}_\rho} + \| \omega\|_{H^{k+1}(\lambda d(x,\partial\Omega) \ge \delta_0+\delta)}
. \end{equation}
For the heat semigroup, we may integrate by parts in $\wtd \theta$ or $\wtd z$. It follows directly from the representation of the Green function that derivatives of the semigroup $\nabla_{\wtd \theta, \wtd z} e^{\nu \wtd t S_H} $ are of order $(\nu \wtd t)^{-1/2}$ of the semigroup itself. Therefore, the first-order derivative term in $\wtd R_\Delta$ can be treated systematically as follows:
$$
\begin{aligned}
\nu \lambda^2 \Big\| \int_{0}^{\wtd t}e^{\nu(\wtd t-\wtd t')S_H} \nabla_{\wtd \theta, \wtd z } h(\wtd t')  \; d\wtd t' \Big\|_{\cW^{k,1}_\rho}
&\lesssim \nu \lambda^2 \int_{0}^{\wtd t}(\nu (\wtd t - \wtd t'))^{-1/2} \| h(\wtd t')\|_{\cW^{k,1}_\rho}   \; d\wtd t'
\\&\lesssim \sqrt \nu \lambda^2 \sup_{0\le \wtd t' \le \wtd t} \| h \|_{\cW^{k,1}_\rho} .
\end{aligned}$$
The zero-order term is treated similarly. The analysis doesn't apply directly to the second-order derivative term $\del^2_{\wtd \theta} [ \wtd m \wtd \omega]$ due to the singularity in time $(\nu t)^{-1}$, if integration by parts was to perform twice. However, in the Fourier variable $\alpha$, we compute
$$\nu \lambda^2\int_{0}^{\wtd t} (e^{\nu(\wtd t-\wtd t')S_H} \del^2_{\wtd \theta} [ \wtd m \wtd \omega])_\alpha (\wtd t') \; d\wtd t'
= \nu \alpha^2 \lambda^2\int_{0}^{\wtd t} \int_0^\infty H_\alpha(t,\wtd y;\wtd z) [ \wtd m \wtd \omega]_\alpha (\wtd t') \; d \wtd yd\wtd t'  .
$$
Observe that the Green kernel $H_\alpha$ has the diffusion term $e^{-\nu\alpha^2 \wtd t}$, for which we use
$$ \nu \alpha^2 \lambda^2\int_{0}^{\wtd t} e^{-\nu \alpha^2 (\wtd t - \wtd t') } d \wtd t' \lesssim \lambda^2$$
yielding the claim \eqref{claim-S1}.

Next, we claim that
\begin{equation}\label{claim-S2}
\Big \|  \nu \lambda^2\int_{0}^{\wtd t}e^{\nu(\wtd t-\wtd t')S_R} \wtd R_\Delta\wtd \w(\wtd t') \; d\wtd t'\Big\|_{\cW^{k,1}_\rho} \lesssim  \nu \lambda^2\int_{0}^{\wtd t} \| \partial_{\wtd \theta}\omega(\wtd t) \|_{\cW^{k,1}_\rho} \; d\wtd t+ \| \omega\|_{H^{k+1}(\lambda d(x,\partial\Omega) \ge \delta_0+\delta)}
. \end{equation}
It suffices to check for the stationary Green kernel $\mu_f e^{-\theta_0\mu_f(\wtd y+\wtd z)} $ and for the second-order derivative term $ \del^2_{\wtd \theta} [ \wtd m \wtd \omega]$ appearing in $ \wtd R_\Delta\wtd \w(\wtd t')$. For this term, we make use of the fact that $\wtd m$ vanishes at $\wtd z =0$; namely, we can write $\wtd m = \wtd z \wtd m_1$ and use
$\mu_f e^{-\theta_0\mu_f \wtd z} \wtd z \lesssim 1$, which controls one spatial derivative, since $\mu_f = |\alpha| + \nu^{-1/2}$. This proves the claim \eqref{claim-S2}.

Finally, putting the previous bounds together into the Duhamel representation \eqref{Duh-w0}, we have obtained
\begin{equation}\label{bootstrap-Stokes}
\begin{aligned}
\| \omega(\wtd t)\|_{\cW^{k,1}_\rho} &\lesssim \| \omega_{0}\|_{\cW^{k,1}_\rho}  + \| \omega_0\|_{H^{k+1}(\lambda d(x,\partial\Omega) \ge \delta_0+\delta)}
\\&\quad + \lambda^2 \sup_{0\le \wtd t' \le \wtd t} \| \omega (\wtd t') \|_{\cW^{k,1}_\rho} +  \nu \lambda^2\int_{0}^{\wtd t} \| \partial_{\wtd \theta}\omega(\wtd t) \|_{\cW^{k,1}_\rho} \; d\wtd t
\\&\quad + \| \omega\|_{H^{k+1}(\lambda d(x,\partial\Omega) \ge \delta_0+\delta)}
\end{aligned}\end{equation}
for any $k\ge 0$. The standard energy estimates for the heat equation (away from the boundary) yield
\begin{equation}\label{standard-heat}\| \omega\|_{H^{k+1}(\lambda d(x,\partial\Omega) \ge \delta_0+\delta)} \lesssim \| \omega_0 \|_{H^{k+1}(\lambda d(x,\partial\Omega) \ge \delta_0/2)}.
\end{equation}
It remains to treat the third and forth terms on the right hand side of \eqref{bootstrap-Stokes}. We bound these terms by iteration,
introducing $$
\begin{aligned}
A_0(\beta) : =&\quad \sup_{0\le  k\le 4} \big(  \sup_{0<\beta \wtd t< \rho_0} \sup_{0<\rho<\rho_0-\beta \wtd t} \Bigl{\{}
 \| \omega(\wtd t)\|_{\mathcal{W}^{k,1}_\rho} +  \|\partial_{\wtd \theta}\omega(\wtd t)\|_{\mathcal{W}^{k,1}_\rho}(\rho_0 - \rho- \beta \wtd t)^{\zeta}\Bigr{\}}\big) \end{aligned}
$$
for some $\zeta \in (0,1)$.
We bound
$$\begin{aligned}
 \nu \lambda^2\int_{0}^{\wtd t} \| \partial_{\wtd \theta}\omega(\wtd t) \|_{\cW^{k,1}_\rho} \; d\wtd t
&\le C_0\nu \lambda^2 A_0(\beta) \int_0^{\wtd t}  (\rho_0-\rho-\beta \wtd s)^{-\zeta}\; d\wtd s
\\&\le
C_0\nu \lambda^2 \beta^{-1}A_0(\beta) .
\end{aligned}$$
Next, we check the bound on $ \|\partial_{\wtd \theta}\omega(\wtd t)\|_{\mathcal{W}^{k,1}_\rho}$. We focus only the worst term as in \eqref{claim-S2}. Note that $\rho < \rho_0 - \beta \wtd t \le \rho_0 - \beta \wtd s$. Thus, we take $\rho' = \frac{\rho + \rho_0 - \beta s}{2}$ and bound
$$
\begin{aligned}
\Big \| & \nu \lambda^2 \partial_{\wtd \theta}\int_{0}^{\wtd t}e^{\nu(\wtd t-\wtd t')S_R} \wtd R_\Delta\wtd \w(\wtd t') \; d\wtd t'\Big\|_{\cW^{k,1}_\rho}
\\&\lesssim  \nu \lambda^2\int_{0}^{\wtd t}  \frac{1}{\rho' - \rho} \| \partial_{\wtd \theta}\omega(\wtd t) \|_{\cW^{k,1}_{\rho'}} \; d\wtd t+ \| \omega\|_{H^{k+1}(\lambda d(x,\partial\Omega) \ge \delta_0+\delta)}
\\&\le C_0 \nu \lambda^2\int_0^{\wtd t}  (\rho_0-\rho-\beta s)^{-1-\zeta}\; ds + \| \omega_0 \|_{H^{k+1}(\lambda d(x,\partial\Omega) \ge \delta_0/2)}
\\&\le C_0 \nu \lambda^2 \beta^{-1} A_0(\beta) (\rho_0-\rho-\beta \wtd t)^{-\zeta} + \| \omega_0 \|_{H^{k+1}(\lambda d(x,\partial\Omega) \ge \delta_0/2)} .
\end{aligned}$$
This proves that
$$ A_0(\beta ) \lesssim  \| \omega_{0}\|_{\cW^{k,1}_\rho}  + \| \omega_0\|_{H^{k+1}(\lambda d(x,\partial\Omega) \ge \delta_0/2)}  + \Big(\lambda^2 + \nu \lambda^2 \beta^{-1}\Big) A_0(\beta).$$
Taking $\lambda$ and $\nu$ small, the last term can be absorbed into the left hand side, completing the bounds on $A_0(\beta)$ or the $\cW^{k,1}_\rho$ norm for the vorticity. Note that we do not require $\beta$ to be sufficiently large (compared with the nonlinear iteration provided in the next section). As a consequence, the proposition holds for any given finite time.
\end{proof}

\section{Nonlinear analysis}\label{sec-nonlinear}

As already mentioned in the introduction, we construct the solutions to the Navier-Stokes equation via the vorticity formulation
\beq\label{NS-vor1}
\begin{aligned}
\partial_{t}\omega + u\cdot\nabla\omega =\nu\Delta\omega
\end{aligned}\eeq
together with the nonlocal boundary condition \eqref{BC-vor} and {with initial data $\omega_{\vert_{t=0}}=\omega_0$ satisfying 
\begin{equation}\label{assumption-data} \|\omega_0\|_{\mathcal{W}^{2,1}_\rho}  +  \|\omega_0\|_{H^4(\{\lambda d(x,\partial\Omega)\ge \delta_0/2\})} < \infty.\end{equation}}
Introduce the smooth cutoff function $\phi^b$ as in \eqref{def-phib}, and write
\begin{equation}\label{decomp-omega} \omega = \omega^b + \omega^i, \qquad \omega^b = \phi^b \omega, \qquad \omega^i = (1-\phi^b)\omega.\end{equation}
We also define the corresponding velocity field through the Biot-Savart law
\begin{equation}\label{decomp-u}
 u = u^b + u^i, \qquad u^b = \nabla^\perp \Delta^{-1} \omega^b, \qquad u^i = \nabla^\perp \Delta^{-1} \omega^i.
 \end{equation}
This yields
\begin{equation}\label{vorticity-bdry}
\left\{
\begin{aligned}
\pt_t \w^b + u\cdot\nabla \w^b &= \nu\triangle \w^b
\\
\nu (\partial_n + DN)\omega^b_{\vert_{\partial \Omega}} &= [\partial_n \Delta^{-1} ( u \cdot \nabla \omega)]
_{\vert_{\partial \Omega}}
\end{aligned}\right.
\end{equation}
for the vorticity near the boundary, and
\begin{equation}\label{vorticity-away}
\left\{
\begin{aligned}
\pt_t \w^i + u\cdot \nabla \omega^i &=\nu\triangle \w^i
\\
\w^i_{\vert{\partial\Omega}} & = 0
\end{aligned}\right.
\end{equation}
for the vorticity away from the boundary. Here, we note that the boundary condition on $\omega^i$ follows directly from the definition \eqref{decomp-omega}, while the boundary condition on $\omega^b$ was due to the fact that $DN \omega^i = 0$ by Lemma \ref{lem-DN}. We also note that the velocity field $u$ that appears in both the systems is the full velocity, which is the summation of $u^b$ and $u^i$ generated by $\omega^b$ and $\omega^i$, respectively.

We shall construct the near boundary vorticity solving \eqref{vorticity-bdry} through the semigroup of the Stokes problem. Indeed, we have the following standard Duhamel's integral representation, written in the scaled variables,
 \beq\label{Duh-w}
\wtd \omega(\wtd t)=e^{\nu \wtd t S}\wtd \omega_{0} +\int_{0}^{\wtd t}e^{\nu(\wtd t-\wtd t')S} f(\wtd t') \; d\wtd t' + \int_0^{\wtd t} \Gamma(\nu (\wtd t-\wtd t'))g(\wtd t') \;d\wtd t'
\eeq
where
\begin{equation}\label{forcingStokes1}
\begin{aligned}
f(\wtd t) =   - \lambda^{-2} u \cdot \nabla \omega^b ,
\qquad g(\wtd t )& =  \lambda^{-1}[\partial_n \Delta^{-1} ( u \cdot \nabla \omega)]
_{\vert_{\partial \Omega}} .
\end{aligned}\end{equation}
Here, $e^{\nu \wtd t S}$ denotes the semigroup of the corresponding Stokes problem and $\Gamma(\nu \wtd t)$ being its trace on the boundary; see Section \ref{sec-realStokes}.

\subsection{Global Sobolev-analytic norm}\label{sec-nonlinear}

We now introduce Sobolev-analytic norms to control global vorticity. Let us fix positive numbers $\rho_0, \delta_0,$ and $\zeta\in (0,1)$. Introduce the following family of nonlinear iterative norms for vorticity:
\beq\label{def-normw}
\bega
A(\beta) : =&\quad  \sup_{0<\lambda^2\beta t< \rho_0} \Big[ \sup_{0<\rho<\rho_0-\beta \lambda^2 t} \Bigl{\{}
 \| \omega(t)\|_{\mathcal{W}^{1,1}_\rho} +  \|\omega(t)\|_{\mathcal{W}^{2,1}_\rho}(\rho_0 - \rho- \lambda^2\beta t)^{\zeta}\Bigr{\}}
 \\&\quad +  \|\omega(t)\|_{H^4(\{\lambda d(x,\partial\Omega)\ge \delta_0/2\})} \Big]
\enda
\eeq
for a parameter $\beta>0$, with recalling $$ \| \omega(t)\|_{\mathcal{W}^{k,1}_\rho}  = \sum_{j+\ell \le k}  \|\partial_{\wtd \theta}^j (\wtd z\partial_{\wtd z})^\ell \omega(t)\|_{\cL^1_\rho}.$$
Note that by definition the norm $\|\cdot\|_{\cW^{k,1}_\rho}$ controls the analyticity of the vorticity near the boundary, precisely in the region
$\lambda d(x,\partial\Omega)\le \delta_0+\rho,$ while the $H^4$ norm is to control the Sobolev regularity away from the boundary. We shall show that the vorticity norm remains finite for sufficiently large $\beta$. The weight $(\rho_0-\rho-\lambda^2\beta t)^\zeta$, with a small $\zeta>0$, is standard in the literature to avoid time singularity when recovering the loss of derivatives (\cite{Asano, Caflisch}). See also \cite{GrenierNguyen3} for an alternative framework to construct analytic solutions through generator functions.

Our goal is to prove the following key proposition.

\begin{proposition}\label{prop-key} For $\beta>0$, there holds
$$ A(\beta) \le  C_0 \|\omega_0\|_{\mathcal{W}^{2,1}_\rho}  +  C_0\|\omega_0\|_{H^4(\{\lambda d(x,\partial\Omega)\ge \delta_0/2\})} + C_0 \beta^{-1} A(\beta)^2 . $$
\end{proposition}

{

In Section \ref{sec-proof}, we will show that our main theorem, Theorem \ref{theo-main}, follows straightforwardly from Proposition \ref{prop-key}.
}

\subsection{Analytic bounds near the boundary}

In this section, we bound the vorticity near the boundary $\lambda d(x,\partial \Omega) \le \delta_0 + \rho_0$, on which by definition $\omega = \omega^b$ and therefore the Duhamel representation \eqref{Duh-w} holds. Let $\rho < \rho_0 - \lambda^2\beta t$. Recalling the notation $\wtd t = \lambda^2 t$ and using \eqref{Duh-w}, we bound
\begin{equation}\label{est-wbdry}
\| \wtd \omega(\wtd t)\|_{\cW_{\rho}^{k,1}} \le \|e^{\nu \wtd t S}\wtd \omega_{0}\|_{\cW_{\rho}^{k,1}} + \int_{0}^{\wtd t} \| e^{\nu(\wtd t-\wtd t')S} f(\wtd t')\|_{\cW_{\rho}^{k,1}} \; d\wtd t' + \int_0^{\wtd t} \|\Gamma(\nu (\wtd t-\wtd t'))g(\wtd t')\|_{\cW_{\rho}^{k,1}} \;d\wtd t'
\end{equation}
for $0<k\le  4$ and for $f,g$ defined as in \eqref{forcingStokes1}. Let us bound each term on the right. Using the semigroup estimates, Proposition \ref{prop-Stokes}, we  have
$$
\begin{aligned}
 \| e^{\nu \wtd t S} \wtd \omega_0\|_{\cW^{k,1}_\rho} &\le C_0 \| \wtd \omega_0\|_{\cW^{k,1}_\rho} +\| \wtd z \wtd \omega_0\|_{H^{k+1}(\wtd z\ge \delta_0 + \rho)}
 \\
 &\le C_0 \| \wtd \omega_0\|_{\cW^{k,1}_\rho} + \| \omega_0\|_{H^{k+1}(\lambda d(x,\partial\Omega)\ge \delta_0 + \rho)} .
 \end{aligned}$$
While for the second integral term in \eqref{est-wbdry}, we have
 $$ \int_{0}^{\wtd t} \| e^{\nu(\wtd t-\wtd t')S} f(\wtd t')\|_{\cW_{\rho}^{k,1}} \; d\wtd t' \lesssim  \int_{0}^{\wtd t}  \Big[ \| f(\wtd t')\|_{\cW_{\rho}^{k,1}} + \|\wtd z f(\wtd t') \|_{H^{k+1}(\wtd z \ge \delta_0 + \rho)} \Big] \; d\wtd t' .$$
Then, we use \eqref{forcingStokes1}, in the above formula with $f(\wtd t)$ replaced by $- \lambda^{-2} u \cdot \nabla \omega^b . $ First, using the standard elliptic theory for $k = 0,1,2$, we bound
$$
\begin{aligned}
\|\wtd z (u \cdot \nabla \omega^b)(\wtd t') \|_{H^{k+1}(\wtd z \ge \delta_0 + \rho)} &\lesssim  \| \omega \|_{H^4(\{\lambda d(x,\partial\Omega)\ge \delta_0 /2\})}^2
\lesssim A(\beta)^2.
\end{aligned}$$
Next, for  the analytic norm, with the bilinear estimates from Lemma \ref{lem-bilinear}, we have:
$$
\begin{aligned}
 \| u\cdot \nabla \omega^b \|_{\cL^1_\rho} &\le C \Big( \| \omega\|_{\cL^1_\rho} + \|\omega\|_{H^{1}(\{\lambda d(x,\partial\Omega)\ge \delta_0\})} \Big)\| \partial_{\wtd \theta}\omega^b\|_{\cL^1_\rho}
 \\& \quad + C \Big ( \|\omega\|_{\cL^1_\rho} + \|\partial_{\wtd \theta}\omega\|_{\cL^1_\rho} + \|\omega\|_{H^{1}(\{\lambda d(x,\partial\Omega)\ge \delta_0\})} \Big) \| \wtd z \partial_{\wtd z}\omega^b\|_{\cL^1_\rho}
 \\& \lesssim \| \omega \|_{\cW^{1,1}_\rho}^2 + \|\omega\|_{H^{1}(\{\lambda d(x,\partial\Omega)\ge \delta_0\})}^2
\\& \lesssim A(\beta)^2
\\
 \| u\cdot \nabla \omega^b \|_{\cW^{1,1}_\rho} & \lesssim \| \omega \|_{\cW^{1,1}_\rho} \| \omega \|_{\cW^{1,2}_\rho} + \|\omega\|_{H^2(\{\lambda d(x,\partial\Omega)\ge \delta_0\})}^2
\\& \lesssim A(\beta)^2 (\rho_0-\rho-\beta \wtd t)^{-\zeta}.
 \end{aligned}$$
Therefore,
$$\begin{aligned}
\int_0^{\wtd t}  \| u\cdot \nabla \omega^b \|_{\cW^{1,1}_\rho} \; d\wtd s
&\le C_0 A(\beta)^2 \int_0^{\wtd t}  (\rho_0-\rho-\beta \wtd s)^{-\zeta}\; d\wtd s
\\&\le
C_0 \beta^{-1} A(\beta)^2 .
\end{aligned}$$
Similarly, we consider the case when $k=2$. Noting $\rho < \rho_0 - \beta t \le \rho_0 - \beta s$, we take $\rho' = \frac{\rho + \rho_0 - \beta s}{2}$ and compute
$$\begin{aligned}
 \int_0^t \| u\cdot \nabla \omega^b \|_{\cW^{2,1}_\rho} \; ds
&\le C_0 \int_0^t\frac{1}{\rho' - \rho} \| u\cdot \nabla \omega^b \|_{\mathcal{W}^{1,1}_{\rho'}}\; ds
\\&\le C_0 A(\beta)^2 \int_0^t  (\rho_0-\rho-\beta s)^{-1-\zeta}\; ds
\\&\le C_0 \beta^{-1} A(\beta)^2 (\rho_0-\rho-\beta t)^{-\zeta}.
\end{aligned}$$

Finally, we treat the last integral term in \eqref{est-wbdry}. Precisely, we will show that, for $k\le 2$:
\beq\label{bdr term}
\bega
\|\Gamma(\nu(\wtd t-\wtd t'))g(\wtd t')\|_{\mathcal W_\rho^{k,1}}\le & \quad C_0 \|u\cdot\nabla\w^b(\wtd t')\|_{\mathcal W^{k,1}_\rho }
+C_0\|\w(\wtd t')\|^2_{H^4(\lambda d(x,\pt\Omega)\ge \delta_0/2)}
\\&\quad +C_0\|\w(\wtd t')\|_{\mathcal W^{k,1}_\rho}\|\w(\wtd t')\|_{H^4(\lambda d(x,\pt\Omega)\ge \delta_0/2)}
\enda 
\eeq
which would then imply 
\[
\int_0^{\wtd t} \|\Gamma(\nu (\wtd t-\wtd t'))g(\wtd t')\|_{\cW_{\rho}^{2,1}} d \wtd t' \le C_0\lw(A(\beta)^2 + \beta^{-1} A(\beta)^2 (\rho_0-\rho-\beta t)^{-\zeta}\rw) .
\]
Here the constant $C_0$ may change from line to line. It remains to give the proof for the inequality \eqref{bdr term}.
First, by Proposition \ref{prop-Stokes}, we have 
\[
\|\Gamma(\nu(\wtd t-\wtd t'))g(\wtd t')\|_{\mathcal W_\rho^{k,1}}\le C_0\sum_{\al}|\al|^k|g_\al|e^{\eps_0(\delta_0+\rho)|\al|}, 
\]
where $g_\al$ is given by 
\[
g_\al=\lambda^{-1}\pt_{n}\triangle^{-1}(u\cdot\nabla\w)_{\al}|_{\pt\Omega} .
\]
Let $\Phi=\triangle^{-1}(u\cdot\nabla \w)$.  By definition, $\Phi$ solves
\[\begin{cases}
&\triangle \Phi=u\cdot\nabla \w,\qquad x\in \Omega\\
&\Phi|_{\pt\Omega}=0.
\end{cases}
\]
In the rescaled geodesic coordinates, we have $g_{\al}=\pt_{\wtd z}\Phi_{\al}(0)$.
Let $\Phi^b=\Phi(x)\phi^b(x)$, we have 
\[\begin{cases}
&\triangle \Phi^b=2\nabla_x\phi^b\cdot\nabla_x\Phi^b+\triangle\phi^b \Phi+\phi^b u\cdot\nabla \w\\
&\Phi^b|_{z=0}=0.
\end{cases} 
\]
By a direct calculation, we have
\[\bega 
&e^{\eps_0(\delta_0+\rho)|\al|}g_\al(\wtd t')=\pt_{z}\Phi_{\al}^b|_{\wtd z=0}\\
&=\int_0^\infty e^{|\al|(\eps_0(\delta_0+\rho)-\wtd z)}\lw\{\lambda^2 \lw(\wtd R_{\triangle}\wtd \Phi_b\rw)_\al (\wtd z)-\lambda^{-2}\lw(2\nabla_x\phi^b\cdot\nabla_x \Phi^b-\Phi\triangle\phi^b-\phi^b u\cdot\nabla \w\rw)_\alpha\rw\}d\wtd z\\
&=I_{1,\al}+I_{2,\al}+I_{3,\al}+I_{4,\al}.
\enda 
\]
\\~
\textbf{Treating $I_{1,\al}$.}~As in the proof of Proposition \ref{exp-realSW} for $\wtd R_{\triangle}$, we have 
\[\bega
|I_{1,\al}|\le &C_0 |\al|^2 \lambda^2 \int_0^\infty  e^{|\al|(\eps_0(\delta_0+\rho)-\wtd z)} |\wtd z \Phi^b_{\al}(\wtd z)|d\wtd z\\
&+C_0\lambda^2 \int_0^\infty e^{|\al|(\eps_0(\delta_0+\rho)-\wtd z)} \lw(|\al||\Phi^b_{\al}(\wtd z)|+|\pt_{\wtd z}\Phi_\al^b|\rw)d\wtd z .
\enda 
\]
First, we use the inequality $\wtd z|\al|e^{-|\al|\wtd z}\le e^{-\frac{1}{2}|\al|\wtd z}$ to get 
\[\bega 
|I_{1,\al}|&\le C_0\lambda^2 \int_0^\infty e^{|\al|\lw(\eps_0(\delta_0+\rho)-\frac{1}{2}\wtd z\rw)} \lw(|\al||\Phi^b_{\al}(\wtd z)|+|\pt_{\wtd z}\Phi_\al^b|\rw)d\wtd z\\
&\le C_0\lambda^2  \int_0^{\delta_0+\rho} e^{|\al|\eps_0(\delta_0+\rho-\wtd z)} \lw(|\al||\Phi^b_{\al}(\wtd z)|+|\pt_{\wtd z}\Phi_\al^b|\rw)d\wtd z\\
&\quad +C_0\lambda^2\int_{\delta_0+\rho}^\infty \lw(|\al||\Phi^b_{\al}(\wtd z)|+|\pt_{\wtd z}\Phi_\al^b|\rw)d\wtd z .
\enda 
\]
For the first term, we use the $L^1_\mu$ elliptic estimate for the velocity (since the kernel $K_\al\in L^1$), to get 
\beq\label{ineq1}
\bega 
&\sum_{\al}|\al|^k \int_0^{\delta_0+\rho} e^{|\al|\eps_0(\delta_0+\rho-\wtd z)} \lw(|\al||\Phi^b_{\al}(\wtd z)|+|\pt_{\wtd z}\Phi_\al^b|\rw)d\wtd z\\
&\le C \|\phi^b u\cdot\nabla \w\|_{\mathcal W_\rho^{k,1}}+C\|\Phi\|_{H^{k+2}(\lambda d(x,\pt\Omega)\ge \delta_0+\rho_0)}.\\
\enda\eeq
Now we have 
\beq\label{ineq3}
\bega 
\|\phi^bu\cdot\nabla \w\|_{\mathcal W^{k,1}_\rho}&=\|u\cdot\nabla \w^b-(u\cdot\nabla\phi^b)\w\|_{\mathcal W^{k,1}_\rho}\\
&\le C\lw(\|u\cdot\nabla\w^b\|_{\mathcal W^{k,1}_\rho}+\|u \w\|_{H^{k}(\lambda d(x,\pt\Omega)\ge \delta_0)}\rw)\\
&\le C \| u\cdot\nabla \w^b\|_{\mathcal W_\rho^{k,1}}+C\|\w\|_{H^4(\lambda d(x,\pt\Omega)\ge \delta_0/2)}\lw(\|\w\|_{H^4(\lambda d(x,\pt\Omega)\ge \delta_0/2)}+\|\w\|_{\mathcal W^{k,1}_\rho}\rw).
\enda 
\eeq
By standard elliptic estimate, we have 
\beq\label{ineq2}
\bega 
\|\Phi\|_{H^{k+2}(\lambda d(x,\pt\Omega)\ge \delta_0+\rho_0)} \le &C\|\phi^b u\cdot\nabla\w\|_{\mathcal W_\rho^{k,1}}+C\|u\cdot\nabla\w\|_{H^{k}(\lambda d(x,\pt\Omega)\ge \delta_0)}\\
\le &  C \| u\cdot\nabla \w^b\|_{\mathcal W_\rho^{k,1}}+\|\w\|^2_{H^4(\lambda d(x,\pt\Omega)\ge \delta_0/2)}+\|\w\|_{H^4(\lambda d(x,\pt\Omega)\ge \delta_0/2)}\|\w\|_{\mathcal W^{k,1}_\rho}.
\enda 
\eeq
Combining \eqref{ineq1},\eqref{ineq2} and \eqref{ineq3}, we have 
\[
\sum_{\al}|\al|^k |I_{1,\al}|\le C_0\lw(\|u\cdot\nabla\w^b(\wtd t')\|_{\mathcal W^{k,1}_\rho }+\|\w\|^2_{H^4(\lambda d(x,\pt\Omega)\ge \delta_0/2)}+\|\w\|_{H^4(\lambda d(x,\pt\Omega)\ge \delta_0/2)}\|\w\|_{\mathcal W^{k,1}_\rho}
\rw)\]
as claimed in \eqref{bdr term}. The proof for $I_{1,\al}$ is complete.~\\

~\\
\textbf{Treating $I_{2,\al}$.}~For $I_{2,\al}$, we note that the domain of integration is $\wtd z\ge \delta_0+\rho_0>\delta_0+\rho$, we have 
\[
|\al|^k e^{|\al|(\eps_0(\delta_0+\rho)-\wtd z)}\le C.
\]
Thus we have 
\[\bega 
\sum_{\al}|\al|^k|I_{2,\al}|&\le C\sum_{\al} \|\nabla_x \Phi^b_{\al}\|_{L^1(\wtd z\ge \delta_0+\rho_0)}\le C\|d(x,\pt\Omega)\nabla_x \Phi\|_{H^1(\lambda d(x,\pt\Omega)\ge \delta_0+\rho_0)}\\
&\le C\|d(x,\pt\Omega)\Phi\|_{H^2(\lambda d(x,\pt\Omega)\ge \delta_0+\rho_0)}\\
&\le C \|\phi^b u\cdot\nabla\w\|_{\mathcal W_{\rho}^{k,1}}+C\|u\cdot\nabla\w\|_{L^2(\lambda d(x,\pt\Omega)\ge \delta_0)},
\enda 
\]
which is bounded by the right hand side of \eqref{bdr term}. The proof for $I_{2,\al}$ is complete.

~\\
\textbf{Treating $I_{3,\al}$.}~Similarly, for $I_{3,\al}$, we get 
\[\bega 
\sum_{\al}|\al|^k|I_{3,\al}|&\le C \|d(x,\pt\Omega)\Phi\|_{H^1(\lambda d(x,\pt\Omega)\ge \delta_0+\rho_0)}\\
&\le C\|\phi^b u\cdot\nabla\w\|_{\mathcal W^{k,1}_\rho}+ C\|u\cdot\nabla\w\|_{L^2(\lambda d(x,\pt\Omega)\ge \delta_0)}.
\enda 
\]
This is also bounded by the right hand side of \eqref{bdr term}. The proof for $I_{3,\al}$ is complete.

~\\
\textbf{Treating $I_{4,\al}$.}~For $I_{4,\al}$ we have 
\[
\sum_{\al}|\al|^k |I_{4,\al}|\le \|\phi^b u\cdot\nabla\w\|_{\mathcal W^{k,1}_\rho} .
\]
We rewrite $\phi^bu\cdot\nabla \w=u\cdot\nabla(\phi^b\w)-u\cdot\nabla\phi^b \w=u\cdot\nabla \w^b-(u\cdot\nabla\phi^b)\w$. Hence we obtain 
\[\bega 
\sum_{\al}|\al|^k|I_{4,\al}|&\le C\lw(\|u\cdot\nabla\w^b\|_{\mathcal W^{k,1}_\rho}+\|u \w\|_{H^{k+1}(\lambda d(x,\pt\Omega)\ge \delta_0+\rho_0)}\rw)\\
&\le C \|u\cdot\nabla\w^b\|_{\mathcal W^{k,1}_\rho}+C\|\w\|^2_{H^4(\lambda d(x,\pt\Omega)\ge \delta_0/2)} +C \|\w\|_{\mathcal W^{k,1}_\rho}\|\w\|_{H^4(\lambda d(x,\pt\Omega)\ge \delta_0/2)}.\enda 
\]
This completes the bound for $I_{4,\al}$. ~\\

Combining all of the above, we obtain bounds on $A(\beta)$ in the analytic norm.

\subsection{Sobolev bounds away from the boundary}
Finally, we bound the vorticity away from the boundary. Recall that
\begin{equation}\label{vorticity-away}
\left\{
\begin{aligned}
\pt_t \w^i + u\cdot \nabla \omega^i &=\nu\triangle \w^i
\\
\w^i_{\vert{\partial\Omega}} & = 0
\end{aligned}\right.
\end{equation}
Note that by definition, $\omega^i$ vanishes in the region when $\lambda d(x,\partial \Omega) \le \delta_0$. We perform the standard energy estimates, for $k\ge 3$ so that the standard Sobolev embedding applies, yielding
$$
\begin{aligned}
\frac{d}{dt}\| \omega^i \|_{H^k}^2 + \nu \| \nabla \omega^i \|_{H^k}^2 &\lesssim \| u \|_{H^k} \| \omega^i\|_{H^k}^2
\\
&\lesssim \| \omega^i\|^3_{H^k} + \| u^b\|^3_{H^k(\lambda d(x,\partial \Omega)\ge \delta_0)}.
\end{aligned}$$
Using the elliptic theory for the Biot-Savart law $u^b = \nabla^\perp \Delta^{-1} \omega^b$, we have
$$ \| u^b\|_{H^k(\lambda d(x,\partial \Omega)\ge \delta_0)} \lesssim \| \omega^b \|_{\cW^{k,1}_\rho} + \| \omega^b\|_{H^k(\lambda d(x,\partial\Omega) \ge \delta_0)} .$$
This proves that
$$
\begin{aligned}
\frac{d}{dt} \| \omega^i \|_{H^k}^2
&\lesssim \| \omega^b \|^3_{\cW^{k,1}_\rho} + \| \omega^b\|^3_{H^k(\lambda d(x,\partial \Omega)\ge \delta_0)}.
\end{aligned}$$
Integrating in time and recalling the iterative norm $A(\beta)$, we arrive at
$$ \| \omega^i \|_{H^4}^2 \lesssim \| \omega_0 \|_{H^4}^2 + T A(\beta)^2.$$
This bounds the Sobolev norm in $A(\beta)$, completing the proof of Proposition \ref{prop-key}.

{

\subsection{Proof of Theorem \ref{theo-main}}\label{sec-proof}

Finally, we show that our main theorem, Theorem \ref{theo-main}, follows from Proposition \ref{prop-key}. Indeed, taking $\beta$ sufficiently large in Proposition \ref{prop-key}, we obtain uniform bounds on the iterative norm \eqref{def-normw} in term of initial data, which gives the local solution in $\cW^{1,1}_\rho + H^4(\{\lambda d(x,\partial\Omega)\ge \delta_0/2\})$ for $t \in [0,T]$, with $T = \beta^{-1}\lambda^{-2}\rho_0$. In particular, by definition of the iterative norm $A(\beta)$, we have 
$$\|\omega(t)\|_{\mathcal{W}^{1,1}_\rho}  +  \|\omega(t)\|_{H^4(\{\lambda d(x,\partial\Omega)\ge \delta_0/2\})} \le C_0 $$
for $t\in [0,T]$. To prove the stated bound \eqref{vorticity-bd} on vorticity, we note that 
$$
\begin{aligned}
\| \omega\|_{L^\infty(\partial \Omega)} &\lesssim \| \partial_{\wtd z} \omega\|_{\cL_{\rho}^1} + \|\omega(t)\|_{H^2(\{\lambda d(x,\partial\Omega)\ge \delta_0/2\})} .
\end{aligned}$$
It thus suffices to prove that $\| \partial_{\wtd z} \omega\|_{\cL_{\rho}^1} \lesssim \nu^{-1/2}$. Indeed, similar to \eqref{est-wbdry}, we bound  
$$
\| \partial_{\wtd z} \omega(\wtd t)\|_{\cL_{\rho}^{1}} \le \| \partial_{\wtd z} e^{\nu \wtd t S} \omega_{0}\|_{\cL_{\rho}^{1}} + \int_{0}^{\wtd t} \| \partial_{\wtd z} e^{\nu(\wtd t-\wtd t')S} f(\wtd t')\|_{\cL_{\rho}^{1}} \; d\wtd t' + \int_0^{\wtd t} \| \partial_{\wtd z}\Gamma(\nu (\wtd t-\wtd t'))g(\wtd t')\|_{\cL_{\rho}^{1}} \;d\wtd t'
$$
for the same $f,g$ defined as in \eqref{forcingStokes1}. It follows directly from the construction, see Section \ref{sec-realStokes}, that the $\wtd z$-derivative of the semigroup $\partial_{\wtd z} e^{\nu \wtd t S}$ satisfies the same bounds as does $e^{\nu \wtd t S}$, up to an extra factor of $(\nu \wtd t)^{-1/2}$ or $|\partial_{\wtd \theta}| + \nu^{-1/2}$. Therefore,  using the previous bounds on $f(\wtd t)$, we have 
$$
\begin{aligned}
 \int_{0}^{\wtd t} \| \partial_{\wtd z}  e^{\nu(\wtd t-\wtd t')S} f(\wtd t')\|_{\cL_{\rho}^{1}} \; d\wtd t' 
 &\lesssim  \int_{0}^{\wtd t}  (\nu(\wtd t-\wtd t'))^{-1/2}\Big[ \| f(\wtd t')\|_{\cW_{\rho}^{1,1}} + \|\wtd z f(\wtd t') \|_{H^{1}(\wtd z \ge \delta_0 + \rho)} \Big] \; d\wtd t' 
\\
&\lesssim  \int_{0}^{\wtd t}  (\nu(\wtd t-\wtd t'))^{-1/2} \; d\wtd t' 
\\&\lesssim \nu^{-1/2}. 
\end{aligned}$$
Other terms are estimated similarly, giving $\| \partial_{\wtd z} \omega\|_{\cL_{\rho}^1} \lesssim \nu^{-1/2}$ as claimed.  
}

\end{document}